\documentclass[a4paper,DIV=12]{scrartcl}

\usepackage[utf8]{inputenc}
\usepackage[T1]{fontenc}
\usepackage{lmodern}
\usepackage[svgnames,dvipsnames,rgb]{xcolor} %colours
\usepackage{amsmath} 
\usepackage{amssymb}
\usepackage{amsthm}
\usepackage{hyperref}
\hypersetup{linktocpage=true,colorlinks=true,linkcolor=RoyalBlue,citecolor=PineGreen,urlcolor=RoyalBlue}
\usepackage{float}
\usepackage{enumitem}
\usepackage{tikz,pgf}
\usepackage{xspace} 
\usepackage[nameinlink]{cleveref}
\usepackage{algorithm,algorithmicx,algpseudocode}
\usepackage[margin=1in]{geometry}

\newtheorem{theorem}{Theorem}[section]
\newtheorem{lemma}[theorem]{Lemma}
\newtheorem{proposition}[theorem]{Proposition}
\newtheorem{corollary}[theorem]{Corollary}
\newtheorem{claim}[theorem]{Claim}
\newtheorem{conjecture}[theorem]{Conjecture}
\theoremstyle{definition}\newtheorem{definition}[theorem]{Definition}
\theoremstyle{definition}
\theoremstyle{definition}

\crefname{claim}{claim}{claims}

\newcommand{\blue}{\mathrm{blue}}
\newcommand{\red}{\mathrm{red}}
\newcommand{\noStableCutThree}{\mathcal{G}_{sc}}

\DeclareMathOperator{\NAC}{NAC}

\newcommand{\nnac}[1]{\NAC_{\#}(#1)}

\newcommand{\RR}{\mathbb{R}}

\usetikzlibrary{calc}
\colorlet{ecol}{black!50!white}
\definecolor{colR}{rgb}{.932,.172,.172} 
\definecolor{colB}{rgb}{.255,.41,.884} 
\colorlet{colG}{Green!60!white}
\colorlet{colY}{Gold!85!black}

\tikzstyle{vertex}=[circle, draw, fill=black, inner sep=0pt, minimum size=4pt]
\tikzstyle{fvertex}=[circle, draw, fill=white, inner sep=0pt, minimum size=4pt]
\tikzstyle{edge}=[line width=1.5pt,ecol]
\tikzstyle{redge}=[edge,colR]
\tikzstyle{bedge}=[edge,colB]
\tikzstyle{bluecomp}=[edge,decorate,decoration={coil,aspect=0,amplitude=1},colB]
\tikzstyle{redcomp}=[edge,decorate,decoration={coil,aspect=0,amplitude=1},colR]
\tikzstyle{arbitrarygraph}=[dashed,black!70!white,thick]
\tikzstyle{alabelsty}=[black!75,font=\scriptsize]
\tikzstyle{chart}=[draw=colB,fill=colB!50!white,rounded corners=0.5pt]

\title{Stable cuts, NAC-colourings and flexible realisations of graphs}

\author{Katie Clinch\thanks{School of Mathematics and Physics, University of Queensland, Australia. \texttt{k.clinch@uq.edu.au}} \and 
Dániel Garamvölgyi\thanks{HUN-REN-ELTE Egerv\'ary Research Group on Combinatorial Optimization, P\'azm\'any P\'eter s\'et\'any 1/C, 1117 Budapest, Hungary. \texttt{daniel.garamvolgyi@ttk.elte.hu}} \and 
John Haslegrave\thanks{School of Mathematical Sciences, Lancaster University, UK, \texttt{j.haslegrave@lancaster.ac.uk}} \and 
Tony Huynh\thanks{Discrete Mathematics Group, Institute for Basic Science, Daejeon, South Korea \texttt{tony@ibs.re.kr}} \and 
Jan Legerský\thanks{Faculty of Information Technology, Czech Technical University in Prague, Czech Republic, \texttt{jan.legersky@fit.cvut.cz}} \and
Anthony Nixon\thanks{School of Mathematical Sciences, Lancaster University, UK, \texttt{a.nixon@lancaster.ac.uk}}}

\date{}

\begin{document}

\maketitle

\begin{abstract}
    A (2-dimensional) realisation of a graph $G$ is a pair $(G,p)$, where $p$ maps the vertices of $G$ to $\RR^2$. A realisation is flexible if it can be continuously deformed while keeping the edge lengths fixed, and rigid otherwise. We say that $G$ is rigid if every generic realisation of $G$ is rigid; otherwise, $G$ is flexible.

    In this paper, we investigate the relationship between stable cuts and graphs which are either flexible, or admit a flexible (not necessarily generic) realisation with positive edge lengths. We strengthen a result of Chen and Yu, who proved that every $n$-vertex graph with at most $2n-4$ edges has a stable cut, by showing that every flexible graph has a stable cut. The existence of a stable cut is a sufficient, but not necessary, condition for a flexible realisation to exist. Using a result of Le and Pfender on stable cuts, we prove a conjecture of Grasegger, Legersk\'y and Schicho that characterises the minimally rigid graphs which admit a flexible realisation.

    Additionally, we investigate the number of NAC-colourings in various graphs. A NAC-colouring is a type of edge colouring introduced by Grasegger, Legersk\'y and Schicho, who showed that the existence of such a colouring characterises the existence of a flexible realisation with positive edge lengths. We provide an upper bound on the number of NAC-colourings for arbitrary graphs, and construct families of graphs, including rigid and minimally rigid ones, for which this number is exponential in the number of vertices. 
\end{abstract}

%================================================
\section{Introduction}
%================================================

Combinatorial rigidity theory is the study of the relationship between the combinatorial structure of a graph and the rigidity of its realisations in space. A \emph{realisation in $\RR^d$} is a pair~$(G,p)$ consisting of a graph $G$ and a mapping $p \colon V(G) \to \RR^d$. We say that $(G,p)$ is \emph{flexible} if there is a nontrivial continuous motion of the vertices during which the length of each edge remains unchanged; otherwise, $(G,p)$ is \emph{rigid}. Here, \emph{nontrivial} means that at least one distance between non-adjacent vertices changes during the motion. 

A realisation in $\RR^1$ is rigid if and only if its underlying graph is connected. However, in higher dimensions a given graph may have both rigid and flexible realisations, and substantial effort has been devoted to understanding which graphs admit a flexible (or rigid) realisation in a given dimension satisfying various nondegeneracy conditions. One of the most natural such conditions is \emph{quasi-injectivity}, which requires that adjacent vertices are mapped to distinct points, or equivalently, that every edge has a positive length in the realisation. Note that, unlike injectivity, quasi-injectivity is preserved under continuous motions.

For $d \geq 3$, every noncomplete graph has a flexible quasi-injective (in fact, injective) realisation in $\RR^d$. Indeed, we may map two non-adjacent vertices $u,v$ to distinct points in space, and then map all remaining vertices injectively into a line $\ell$ that does not contain either of these points. Rotating $u$ around $\ell$ while keeping $v$ fixed then yields a nontrivial continuous motion. 

In contrast, deciding whether a graph has a flexible quasi-injective realisation in the plane turns out to be a highly nontrivial problem. Nonetheless, Grasegger, Legerský, and Schicho \cite{GLS2019} obtained a remarkably simple and purely combinatorial characterisation of such graphs. A \emph{NAC-colouring} of a graph is a $2$-colouring of its edge set such that both colours are used and every cycle is either monochromatic, or contains at least two edges of both colours. 

\begin{theorem}[{\cite[Theorem 3.1]{GLS2019}}]\label{thm:NACiffflex}
    A connected graph has a flexible quasi-injective realisation in $\RR^2$ if and only if it has a NAC-colouring.
\end{theorem}

Although \Cref{thm:NACiffflex} provides a key tool for investigating graphs with a flexible quasi-injective realisation in the plane
(as well as graphs with a flexible injective realisation in the plane, see \cite{GLSinjective,GLSclassification}), the picture remains far from complete. In fact, deciding whether a graph~$G$ admits a NAC-colouring is NP-complete \cite{Garamvolgyi2022} even in the case when $G$ has maximum degree 5 \cite{LL2024}, and hence to hope for a complete understanding of this graph class would be overly ambitious.

Nonetheless, many natural questions still remain. One prominent problem concerns NAC-colourings of graphs that are generically minimally rigid in the plane.
A realisation $(G,p)$ is \emph{generic} if the coordinates of $p$ are algebraically independent over $\mathbb{Q}$ (that is, they satisfy no nontrivial algebraic relations with rational coefficients). It is a well-known fact that for any graph $G$ and dimension $d$, either all generic realisations of $G$ in $\RR^d$ are rigid, or none of them are. We say that a graph is \emph{(generically) $d$-rigid} if it has a rigid generic realisation in $\RR^d$; otherwise, we say that $G$ is \emph{$d$-flexible}. Thus, a graph being $d$-rigid means that ``most'' of its realisations in $\RR^d$ are rigid. However, a $d$-rigid graph may still have flexible non-generic realisations.
The investigation of these ``paradoxically flexible'' realisations has close ties with the study of linkages, an area whose roots go back to the eighteenth century \cite{Dixon,Kempe1877}.

A graph $G$ is \emph{minimally $d$-rigid} if it is $d$-rigid, but $G-e$ is $d$-flexible for all $e\in E(G)$. It is easy to show that $2$-trees (graphs obtained from an edge by attaching triangles along edges) are minimally $2$-rigid and have no NAC-colourings. It was conjectured in \cite{GLS2019} that they are the only such examples; that is, that every minimally $2$-rigid graph that is not a $2$-tree (or a single vertex) admits a NAC-colouring. This conjecture was confirmed in several special cases in \cite{GLS2019}, but since then has remained open in general. 

Our first main result is an affirmative answer to this conjecture.

\begin{theorem}\label{thm:nac}
    A minimally $2$-rigid graph on at least two vertices has a NAC-colouring if and only if it is not a $2$-tree. 
\end{theorem}

Our proof of \Cref{thm:nac} relies on the notion of stable cuts. A \emph{stable cut} of a graph $G$ is a set of pairwise non-adjacent vertices $X$ such that $G-X$ is disconnected. It was already observed in \cite{GLS2019} that if $G$ has a stable cut, then it admits a NAC-colouring. On the other hand, Le and Pfender \cite{LePf} characterised the class $\noStableCutThree$ of graphs $G$ with $2|V(G)|-3$ edges and no stable cuts (\Cref{thm:stableCutMinusThree} below; we note that a significant gap in \cite{LePf} was recently discovered and fixed by Rauch and Rautenbach \cite{rauch_2024}). Since minimally $2$-rigid graphs satisfy this edge count, the proof of \Cref{thm:nac} reduces to verifying that, except for $2$-trees, every graph in~$\noStableCutThree$ admits a NAC-colouring.

We further explore the relationship between stable cuts and rigidity by showing that every $2$-flexible graph has a stable cut. More precisely, we show that every $2$-connected $2$-flexible graph $G$ has a stable cut that avoids any given vertex (\Cref{prop:flexibleHasStableSeparator}). This strengthens a result of Chen and Yu \cite{ChenYu}, who proved the analogous statement for $2$-connected graphs $G$ with at most $2|V(G)|-4$ edges.
We also give an alternative proof of the fact that $2$-flexible graphs have a stable cut by showing a precise correspondence between stable cuts and \emph{NAP-colourings}, a variant of NAC-colourings related to the existence of quasi-injective flexible realisations on the sphere (\Cref{lem:NAPiffStableCut}). 

As the second main thread of the paper, we investigate the number of NAC-colourings. Just as the existence of a NAC-colouring characterises the existence of a flexible quasi-injective realisation in the plane, the number of NAC-colourings provides a lower bound for the number of distinct motions arising from such realisations. See \Cref{prop:motionbound} for a precise statement.

We prove, among other results, that every graph $G$ has fewer than $4^{|V(G)|}$ NAC-colourings (\Cref{thm:NACupperbound}). This is the first nontrivial improvement (for general graphs) of the naive bound~$2^{|E(G)|}$.
Moreover, for each sufficiently large $n$, we construct an $n$-vertex minimally rigid graph with at least $2.13^n$ NAC-colourings (\Cref{cor:2+}). 

Definitions and preliminary material are given in the next section. We prove \Cref{thm:nac} in \Cref{sec:minimallyrigid}. In \Cref{sec:flexible}, we consider stable cuts in $2$-flexible graphs, and in \Cref{sec:numberofNACs} we investigate questions regarding the number of NAC-colourings. Finally, we highlight some open problems in \Cref{sec:conclusion}.  

%================================================
\section{Preliminaries}
\label{sec:prelims}
%================================================

We begin by reviewing some background from graph theory.  Given a graph $G$, we use $V(G)$ and $E(G)$ to denote the vertex set and the edge set of $G$, respectively.
 If $V'\subseteq V(G)$, then $G[V']$ denotes the vertex-induced subgraph,
and if $E'\subseteq E(G)$, then $G[E']$ denotes the edge-induced subgraph, i.e., the minimal subgraph of $G$ containing $E'$. 
The neighbourhood of a vertex $v \in V(G)$ is denoted by $N_G(v)$.  A \emph{stable set} is a set of pairwise non-adjacent vertices.  We say that $X \subseteq V(G)$ is a \emph{cut} of $G$ if $G-X$ is disconnected, and $X$ is a \emph{stable cut} if it is both a stable set and a cut.   
The complete graph on $n$ vertices is denoted by $K_n$, and the complete bipartite graph on vertex sets of size $n_1$ and $n_2$ by $K_{n_1,n_2}$.
For a positive integer $m$, we let $[m] := \{1,\ldots,m\}$.  

Let $G$ be a graph without isolated vertices.
A \emph{separation} of $G$ is a pair of subgraphs $\{G_1, G_2\}$ of $G$ such that there is a partition $\{E_1, E_2\}$ of $E(G)$ with $G_1=G[E_1]$, $G_2=G[E_2]$, $V(G_1) \setminus V(G_2) \neq \varnothing$, and $V(G_2) \setminus V(G_1) \neq \varnothing$. A separation $\{G_1, G_2\}$ is \emph{stable} if $V(G_1 \cap G_2)$ is a stable set in $G$. Note that if $\{G_1,G_2\}$ is a stable separation, 
then $V(G_1) \cap V(G_2)$ is a stable cut of $G$. 

\subsection{Rigid graphs}

We first define rigid and flexible realisations formally. Let $G$ be a graph. We say that two realisations $(G,p)$ and $(G,q)$ in $\mathbb{R}^d$ are \emph{equivalent} if $\lVert p(u)-p(v)\rVert = \lVert q(u)-q(v)\rVert$ holds for every $uv \in E(G)$, where $\lVert \cdot\rVert$ denotes the Euclidean norm. Similarly, we say that the two realisations are \emph{congruent} if $\lVert p(u)-p(v)\rVert = \lVert q(u)-q(v)\rVert$ holds for all $u,v \in V(G)$.
A \emph{continuous motion} of $(G, p)$ is a continuous function $\phi : (-1, 1) \times V(G) \rightarrow \mathbb{R}^d$ such that $\phi_0 = p$, and such that the realisations $(G,p)$ and $(G,\phi_t)$ are equivalent for all $t \in (-1,1)$, where $\phi_t : V(G) \rightarrow \mathbb{R}^d$ is defined by putting $\phi_t(v) = \phi(t, v)$ for all $v \in V$. 

The continuous motion $\phi$ is \emph{trivial} if the realisations $(G,p)$ and $(G,\phi_t)$ are congruent for all $t \in (-1,1)$. Hence, in a nontrivial motion, the edge lengths are preserved, but the distance between some non-adjacent vertices changes. A realisation is said to be \emph{flexible} if it has a nontrivial continuous motion and \emph{rigid} if it is not flexible. Equivalently, the realisation $(G, p)$ is \emph{rigid} if there exists an $\varepsilon > 0$ such that for any $(G, q)$ equivalent to $(G,p)$ that satisfies $\|p(u)- q(u)\| < \varepsilon$ for all $v \in V(G)$, $(G, q)$ can be obtained from $(G, p)$ by an isometry of $\mathbb{R}^d$.
See \cite{AsimowRoth} for more details.

A graph is called \emph{$d$-rigid} if it has a generic rigid realisation in $\RR^d$ (equivalently, all generic realisations in $\RR^d$ are rigid \cite{AsimowRoth}), and \emph{$d$-flexible} otherwise.
We say that a graph is \emph{minimally $d$-rigid} if it is $d$-rigid but the deletion of any edge yields a $d$-flexible graph.

Throughout the paper, we concentrate on $2$-dimensional realisations and hence we shall refer to $2$-rigid and $2$-flexible graphs as rigid and flexible graphs, respectively. Similarly, we will simply write ``realisation'' instead of ``realisation in $\RR^2$''. See \cite{jordan_2016} for a detailed introduction to combinatorial rigidity theory with an emphasis on the $2$-dimensional case. 

It is a classical observation of Maxwell \cite{Maxwell} that a minimally rigid graph $G$ has exactly $2|V(G)|-3$ edges. Going further,
minimally rigid graphs were characterised in graph-theoretic terms by Pollaczek-Geiringer \cite{Geiringer} and independently by Laman \cite{Laman}. This characterisation leads to a deterministic polynomial time algorithm to test if a graph is minimally rigid (see, for example, \cite{LeeStreinu}).  This algorithm can also efficiently find the \emph{rigid components} of $G$; these are the maximal rigid subgraphs of $G$. It is well-known that the union of two rigid graphs intersecting in at least $2$ vertices is rigid (see~\cite[Theorem 2.2.9]{jordan_2016}, for instance). It follows that each pair of rigid components share at most $1$ vertex. Hence, the rigid components partition the edge set of the graph. 

Given a graph, we call the operation of adding a new vertex and two new edges connecting it to existing vertices a \emph{$0$-extension}. It is well-known that $0$-extensions preserve the property of being (minimally) rigid, see~\cite{jordan_2016}.
In the following, it will be convenient to say that a $0$-extension is \emph{open} if the neighbours of the new vertex are non-adjacent.
We say that a graph is a \emph{$0$-extension graph} if it can be constructed from an edge by a sequence of $0$-extensions. 

A \emph{$2$-tree} is a $0$-extension graph such that no open $0$-extension is used. Equivalently (as defined in the introduction), a $2$-tree is a graph that can be obtained from an edge by attaching triangles along edges. Note that according to our definition, a single edge is a $2$-tree.

\subsection{NAC-colourings}

Let us recall the following definition from the introduction.

\begin{definition}
    \label{def:NAC}
    Let $G$ be a graph. An edge colouring $c\colon  E(G) \rightarrow \{\red, \blue\}$
    is a \emph{NAC-colouring} if it is surjective and every cycle is either monochromatic
    or contains at least two red and two blue edges.
\end{definition}

\begin{figure}[ht]
    \centering
    \begin{tikzpicture}[scale=1.5]
        \node[fvertex] (a) at (0,0) {};
        \node[fvertex] (b) at (1,0) {};
        \node[fvertex] (c) at (0.5,0.5) {};
        \node[fvertex] (d) at (0,1) {};
        \node[fvertex] (e) at (1,1) {};
        \node[fvertex] (f) at (0.5,1.5) {};
        \draw[edge] (a)edge(b) (b)edge(c) (c)edge(a) (d)edge(e) (e)edge(f) (f)edge(d) (a)edge(d) (b)edge(e) (c)edge(f);
    \end{tikzpicture}
    \quad
    \begin{tikzpicture}[scale=1.5]
        \node[vertex] (a) at (0,0) {};
        \node[vertex] (b) at (1,0) {};
        \node[vertex] (c) at (0.5,0.5) {};
        \node[vertex] (d) at (0,1.5) {};
        \node[vertex] (e) at (1,1.5) {};
        \node[vertex] (f) at (0.5,1) {};
        \draw[bedge] (a)edge(b) (b)edge(c) (c)edge(a) (d)edge(e) (e)edge(f) (f)edge(d) ;
        \draw[redge] (a)edge(d) (b)edge(e) (c)edge(f);
    \end{tikzpicture}
    \qquad
    \begin{tikzpicture}[scale=0.75]
        \draw[black!50!white, dashed] (-1.5,0)edge(2.3,0);
        \draw[black!50!white, dashed] (0,1.55)edge(0,-2.1);
        \node[fvertex] (2) at (1.8, 0) {};
        \node[fvertex] (5) at (-1., 0) {};
        \node[fvertex] (7) at (0.7, 0) {};
        \node[fvertex] (1) at (0, -1.75) {};
        \node[fvertex] (6) at (0,  1.2) {};
        \node[fvertex] (4) at (0, -0.8) {};
        \draw[edge]  (6)edge(5) (5)edge(4) (7)edge(4) (7)edge(6);
        \draw[edge] (2)edge(4) (2)edge(6);
        \draw[edge] (1)edge(5) (7)edge(1) ;
        \draw[edge] (2)edge(1);
    \end{tikzpicture}
    \quad
    \begin{tikzpicture}[scale=1.2]
        \node[vertex] (a1) at (-0.5, -0.866025) {};
        \node[vertex] (a2) at (0.5, -0.866025) {};
        \node[vertex] (a3) at (1., 0.) {};
        \node[vertex] (a4) at (0.5, 0.866025) {};
        \node[vertex] (a5) at (-0.5, 0.866025) {};
        \node[vertex] (a6) at (-1.,  0.) {};
        \draw[bedge] (a1)edge(a2) (a1)edge(a4) (a1)edge(a6);
        \draw[redge] (a2)edge(a3) (a2)edge(a5) (a3)edge(a4) (a3)edge(a6) (a5)edge(a4) (a5)edge(a6);
    \end{tikzpicture}
    \quad
    \begin{tikzpicture}[scale=1.2]
        \node[vertex] (a1) at (-0.5, -0.866025) {};
        \node[vertex] (a2) at (0.5, -0.866025) {};
        \node[vertex] (a3) at (1., 0.) {};
        \node[vertex] (a4) at (0.5, 0.866025) {};
        \node[vertex] (a5) at (-0.5, 0.866025) {};
        \node[vertex] (a6) at (-1.,  0.) {};
        \draw[bedge] (a1)edge(a2) (a1)edge(a4) (a2)edge(a3) (a3)edge(a4) (a5)edge(a6);
        \draw[redge] (a1)edge(a6) (a2)edge(a5) (a3)edge(a6) (a5)edge(a4) ;
    \end{tikzpicture}
    \caption{The $3$-prism and $K_{3,3}$ are rigid, but both have flexible realisations (the figures with empty vertices and grey edges).
    The $3$-prism has a unique NAC-colouring modulo swapping colours, while all NAC-colourings of $K_{3,3}$ are isomorphic to the two displayed ones.}
    \label{fig:3prismK33flexible}
\end{figure}
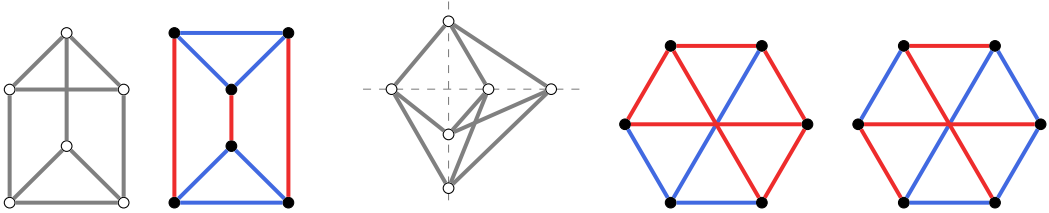
\Cref{fig:3prismK33flexible} shows some examples of NAC-colourings. For convenience, let us also introduce some additional terminology. Let $G$ be a graph and let $c\colon  E(G) \rightarrow \{\red, \blue\}$ be an edge colouring.
A subgraph $H$ of $G$ is \emph{almost red} if exactly one edge of $H$ is blue, and is \emph{almost blue} if exactly one edge of $H$ is red.
We say that $H$ is \emph{almost monochromatic} if it is almost red or almost blue.
The \emph{red components} of $G$ are the components of the graph with vertex set $V(G)$ and edge set the set  of red edges of $G$.
We define the \emph{blue components} of $G$ similarly.
A \emph{monochromatic component} is a red or a blue component.
Using this terminology, $c$ is a NAC-colouring if and only if it is surjective and has no almost monochromatic cycles.
This is equivalent to the condition that $c$ is surjective and
every monochromatic component is an induced subgraph, see~\cite[Lemma~2.4]{GLS2019}.

Note that if $c$ is a NAC-colouring, then the colouring obtained from $c$ by interchanging colours is also a NAC-colouring.  This motivates the following definition.  

\begin{definition}
    For a graph $G$, we let $\nnac{G}$ denote the number of NAC-colourings of $G$ divided by $2$.  
\end{definition}

Let us outline how the parameter $\nnac{G}$ is related to counting motions of a graph $G$.
Given an edge $uv \in E(G)$, a \emph{motion of $G$ with fixed edge $uv$} is a continuous motion of quasi-injective, pairwise equivalent realisations of $G$ such that $u$ and~$v$ are mapped to the same two points
for all realisations in the motion. Note that the choice of the fixed edge is inconsequential in that given a motion of $G$ with fixed edge $uv$, we can obtain a motion with fixed edge $u'v'$ by applying a suitable isometry to each realisation in the original motion.
    
Two motions of $G$ with fixed edge $uv$
are \emph{angle-different} if there is a pair of edges
such that their mutual angle is preserved in one motion, but changes in the other.

\begin{proposition}\label{prop:motionbound}
    For every graph $G$ and $uv \in E(G)$, there are at least $\nnac{G}$ pairwise angle-different motions of $G$ with fixed edge $uv$.
\end{proposition}

\begin{proof}
    The construction used in the proof of \Cref{thm:NACiffflex} in \cite{GLS2019} shows that for every NAC-colouring~$c$,
    there is a motion such that all red edges keep their mutual angles, and all blue ones as well.
    Hence, every pair of distinct NAC-colourings of $G$ yield angle-different motions of $G$, unless they differ only by swapping colours.    
\end{proof}

\section{NAC-colourings of minimally rigid graphs}\label{sec:minimallyrigid}

In this section, we prove \Cref{thm:nac}. As outlined in the introduction, our proof relies on the following easy observation. See \Cref{lem:NAPiffStableCut} below for a proof of (a more precise version of) this statement.

\begin{lemma}[{\cite[Theorem~4.4]{GLS2019}}]
    \label{thm:stableCutImpliesNAC}
    If $G$ is a connected graph with a stable cut, then $G$ has a NAC-colouring.
\end{lemma}

To apply \Cref{thm:stableCutImpliesNAC}, we need to determine which minimally rigid graphs have a stable cut. Fortunately, this has already been done by Le and Pfender \cite{LePf}, with a gap in their proof subsequently filled by Rauch and Rautenbach \cite{rauch_2024}. In fact, their characterisation is valid for the larger family of graphs $G$ with $2|V(G)|-3$ edges. 
Define the family $\noStableCutThree$ as follows:
\begin{enumerate}
    \item the single edge $K_2$ is in $\noStableCutThree$, and
    \item if $G_1\in\noStableCutThree$ and $G_2$ is a $3$-cycle or a $3$-prism, then the graph obtained by gluing $G_1$ and $G_2$ along an edge or a $3$-cycle is in $\noStableCutThree$.
\end{enumerate}

It is not difficult to see that every graph $G \in \noStableCutThree$ has $2|V(G)|-3$ edges and no stable cuts. The main result of \cite{LePf,rauch_2024} shows that the converse is also true.

\begin{theorem}[{\cite[Theorem 5]{LePf}, \cite[Theorem 4]{rauch_2024}}]
    \label{thm:stableCutMinusThree}
    Let $G$ be a graph with $|E(G)|=2|V(G)|-3$.
    Then either $G$ has a stable cut, or $G$ belongs to $\noStableCutThree$.
\end{theorem}

It only remains to show that \Cref{thm:nac} holds for graphs in $\noStableCutThree$.

\begin{proposition}
    \label{prop:noStableCutThreeHaveNAC}
    Every graph in $\noStableCutThree$ that is not a $2$-tree has a NAC-colouring.
\end{proposition}

\begin{proof}
    Let $G$ be a graph in $\noStableCutThree$ that is not a $2$-tree.  Define $\gamma(G)$ to be the minimum number of steps required to construct $G$ in the recursive definition of $\noStableCutThree$. We proceed by induction on~$\gamma(G)$.  The base case is vacuous, since $K_2$ is a $2$-tree.
    Note that the $3$-prism has a (unique, up to swapping colours) NAC-colouring, illustrated in \Cref{fig:3prismK33flexible}.
    
    Suppose that $G$ is obtained from $G_1 \in\noStableCutThree$ and $G_2$ by gluing along an edge $uv$ or a 3-cycle with vertices $\{u,v,w\}$, where $\gamma(G_1)=\gamma(G)-1$ and $G_2$ is either a $3$-cycle or a $3$-prism.
    If~$G_2$ is a $3$-cycle, then $G_1$ cannot be a $2$-tree, since $G$ is not a $2$-tree. 
    By induction, $G_1$ has a NAC-colouring $c_1$,
    and we can extend $c_1$ to a NAC-colouring $c$ of $G$ by setting $c(e)=c_1(uv)$ for all $e\in E(G_2)$.
    On the other hand, if $G_2$ is $3$-prism, then we extend the NAC-colouring of~$G_2$ by colouring all edges in $G_1$
    by the colour of the edge (resp.\ $3$-cycle) along which we are gluing.
\end{proof}

\begin{proof}[Proof of \Cref{thm:nac}]
    Let $G$ be a minimally rigid graph on at least two vertices that is not a $2$-tree. Note that $G$ has $2|V(G)| - 3$ edges.
    If $G$ has a stable cut, then it has a NAC-colouring by \Cref{thm:stableCutImpliesNAC}.
    Otherwise, $G$ is in $\noStableCutThree$ by \Cref{thm:stableCutMinusThree}, and hence
    has a NAC-colouring by \Cref{prop:noStableCutThreeHaveNAC}.
\end{proof}
The proof above actually shows the existence of a NAC-colouring for every graph $G$ with $2|V(G)|-3$ edges,
not only minimally rigid ones. 
However, the case when $G$ is not minimally rigid follows directly from \Cref{thm:NACiffflex},
since then $G$ is flexible.

We can summarize \Cref{thm:NACiffflex,thm:nac} as follows.

\begin{theorem}\label{thm:mainlaman}
    Let $G$ be a minimally rigid graph.
    The following are equivalent:
    \begin{enumerate}
        \item $G$ has a NAC-colouring,
        \item $G$ is not a $2$-tree, and
        \item $G$ has a flexible quasi-injective realisation in the plane.
    \end{enumerate}
\end{theorem}

\Cref{thm:mainlaman} shows that we can decide in polynomial time the existence of a NAC-colouring in minimally rigid graphs. By carefully following the proofs of \Cref{thm:stableCutMinusThree} and \Cref{prop:noStableCutThreeHaveNAC}, one can also find a NAC-colouring in polynomial time
when it exists.

\section{Stable cuts in flexible graphs}
\label{sec:flexible}

The crucial tool in the previous section was \Cref{thm:stableCutMinusThree}, the characterisation of graphs on $n$ vertices and $2n-3$ edges that have no stable cuts. The proof of this theorem in \cite{LePf,rauch_2024} uses the following elegant result of Chen and Yu. 

\begin{theorem}[{\cite[Theorem 1]{ChenYu}}]\label{thm:chenyu}
    Let $G$ be a $2$-connected graph with $|E(G)| \leq 2|V(G)|-4$. Then for every vertex $v \in V(G)$, $G$ has a stable cut avoiding $v$.
\end{theorem}

\Cref{thm:stableCutMinusThree} shows that the bound on the number of edges in \Cref{thm:chenyu} is best possible. In fact, as the number of edges increases, deciding whether there exists a stable cut quickly becomes intractable: it is already NP-complete for the family of graphs with $n$ vertices and at most $(2+\varepsilon)n$ edges, for any positive $\varepsilon$ \cite{LeRanderath}.

To further develop the connection between rigidity and stable cuts, in this section we extend \Cref{thm:chenyu} to the class of flexible graphs. 
We say that a stable cut $S$ \emph{separates} two vertices $u,v$ of $G$ if $u$ and $v$ are in different components of $G - S$.

\begin{theorem}
    \label{prop:flexibleHasStableSeparator}
    Let $G$ be a flexible graph and $u,v \in V(G)$ be such that no rigid component of $G$ contains both $u$ and $v$.
    Then there is a stable cut $S$ of $G$ that separates $u$ and~$v$, and such that every rigid component of $G$ contains at most one vertex of $S$.
    Moreover, if $G$ is $2$-connected, then for every vertex $v \in V(G)$, $G$ has a stable cut avoiding $v$.
\end{theorem}

\begin{corollary}\label{cor:flexiblestablecut}
    Every flexible graph has a stable cut.
\end{corollary}

At the end of the section, we give an independent proof for \Cref{cor:flexiblestablecut}, which builds on the relationship between Euclidean and spherical rigidity.

We remark that flexible graphs on $n$ vertices can be dense. Indeed, they can have average degree close to $n$. 
It is easy to see that the higher-dimensional analogue of \Cref{cor:flexiblestablecut} is false.  For example, the $3$-prism is one of many examples that is $3$-flexible, but has no stable cut.

To prove \Cref{prop:flexibleHasStableSeparator}, we will need the following well-known result from rigidity theory. A \emph{vertex split} of a graph $G$ is the following operation: 
choose $v\in V(G)$ and a pair of sets $N_1,N_2$ such that $N_1\cup N_2=N_G(v)$ and $|N_1\cap N_2|=1$, delete $v$ from $G$ and add two new vertices $v_1,v_2$ joined to $N_1,N_2$, respectively, and finally, add the edge $v_1v_2$.

\begin{lemma}[{\cite[Proposition 10]{Wsplit}}]\label{lem:vsplit}
    Let $G$ be a minimally rigid graph and let $G'$ be obtained from $G$ by a vertex split. Then $G'$ is minimally rigid.
\end{lemma}

We will also need some additional concepts from combinatorial rigidity.
The \emph{rank} of a graph $G$ on at least two vertices, denoted by $r(G)$, is $2|V(G)| - 3 - k$, where $k$ is the minimum number of edge additions needed to make $G$ rigid. An \emph{$\mathcal{R}_2$-circuit} is a graph~$C$ on at least two vertices for which $r(C) = r(C-e) = |E(C)|-1$ for all $e \in E(C)$. It is well-known that $\mathcal{R}_2$-circuits are rigid and have minimum degree three, see e.g.~\cite[Section 3.1]{jordan_2016}. Also, for a graph $G$ and a non-adjacent pair of vertices $x,y \in V(G)$, we have $r(G) = r(G+xy)$ if and only if there is an $\mathcal{R}_2$-circuit in $G+xy$ which contains the edge $xy$. 

\begin{proof}[Proof of \Cref{prop:flexibleHasStableSeparator}]
    We prove the statement by induction on $|V(G)|$.
    We may suppose that all rigid components are complete graphs. (When a rigid component is not complete, we add the missing edges; a stable cut in this new graph is also a stable cut in the original graph.)
    Under this assumption, it suffices to show that there is a stable cut separating $u$ and~$v$, since no stable cut can contain two vertices of a rigid component of $G$.
    
    The statement is vacuously true when $|V(G)| = 1$,
    and hence we may suppose that $|V(G)| \geq 2$.
    If the neighbourhood of $u$ is stable, then we take $S$ to be this neighbourhood.
    The hypothesis that $u$ and $v$ are not in the same rigid component implies that $uv$ is not an edge of $G$, and hence $S$ separates $u$ and $v$.
    
    Thus we may suppose that the neighbourhood of $u$ is not stable, or equivalently, that $u$ is in a $3$-cycle $\{u,x_1,x_2\}$.
    For $i \in \{1,2\}$, let $G'_i$ be the graph obtained from $G$ by contracting the edge~$ux_i$ to a vertex $u'_i$. If $u'_i$ and $v$ are in different rigid components of $G'_i$, then by the induction hypothesis, there is a stable cut of $G'_i$ separating $u'_i$ and $v$. This stable cut separates $u$ and $v$ in $G$ as well.

    Suppose, for a contradiction, that $u'_i$ and $v$ are contained in a rigid component of $G'_i$ for both $i \in \{1,2\}$. This means that there is a subset $X'_i \subseteq V(G'_i)$ with $u'_i,v \in X'_i$ and such that $G'_i[X'_i]$ is rigid. Let $X_i \subseteq V(G)$ be the vertex set obtained from $X'_i$ by deleting $u'_i$ and adding $u$ and~$x_i$. Also, let $G_i = G[X_i]$. Note that $G_i$ is flexible, since it contains both $u$ and $v$.
    
    \begin{claim}\label{claim:commonneighbours}
        We have $N_G(u) \cap N_G(x_1) \cap X_1 = \varnothing$ and $N_G(u) \cap N_G(x_2) \cap X_2 = \varnothing$.
    \end{claim}
    
    \begin{proof}
        Fix $i \in \{1,2\}$ and suppose, for a contradiction, that there is a vertex $y \in X_i$ that is a common neighbour of $u$ and $x_i$. It follows that $G_i$ can be obtained from the rigid graph $G'_i[X'_i]$ by a vertex split operation, potentially followed by some edge additions. Hence $G_i$ is rigid, a contradiction. 
    \end{proof}

    Let $H = G_1 \cup G_2$.  
    
    \begin{claim}\label{claim:rankofH}
        We have $r(H) = 2|V(H)| - 4$.
    \end{claim}
    
    \begin{proof}
        Fix $i \in \{1,2\}$. First, note that there is a vertex $y \in N_G(x_i) - N_G(u)$ in $X_i$, since otherwise the rigid graph $G'_i[X'_i]$ would be isomorphic to a subgraph of $G$ that contains both $u$ and $v$. By the same reasoning as in the proof of \Cref{claim:commonneighbours}, we deduce that $G_i + uy$ is rigid, and hence $r(G_i) = 2|V(G_i)| - 4$. Since $\{u,v\}$ is not contained in a rigid subgraph of $G$, we have $r(G_i + uv) = r(G_i) + 1 = 2|V(G_i)| - 3$, that is, $G_i + uv$ is rigid. It follows that $H + uv$ is also rigid, since it can be obtained by gluing two rigid graphs along at least two vertices. Finally, since $H$ is not rigid (as it contains both $u$ and $v$), we have $r(H) = r(H+uv) - 1 = 2|V(H)| - 4$.
    \end{proof}

    If $H+x_1x_2$ is rigid, then $u$ and $v$ are contained in a rigid subgraph of $G$, a contradiction.
    Thus we may assume that $H+x_1x_2$ is flexible. It follows from \Cref{claim:rankofH} that $r(H) = r(H+x_1x_2) = 2|V(H)| - 4$. Hence %$\{x_1,x_2\}$ is $\mathcal{R}_2$-linked in $H$, which means that 
    there is an $\mathcal{R}_2$-circuit $C \subseteq H+x_1x_2$ with $x_1x_2 \in E(C)$. Since $\mathcal{R}_2$-circuits have minimum degree three, there exists a vertex $y \in N_C(x_1) \setminus \{u,x_2\}$. Since $\mathcal{R}_2$-circuits are 
    rigid, we deduce that $u,x_1,x_2$, and $y$ are all in the same rigid component of $G$. It follows from our assumption on $G$ that $\{u,x_1,x_2,y\}$ induces a complete subgraph of $G$, contradicting \Cref{claim:commonneighbours}. This final contradiction finishes the proof of the first part of the theorem. 

        For the second statement, it suffices to prove that if $G$ is $2$-connected, then for all $v \in V(G)$,
    there exists $u\in V(G)$ such that no rigid component contains both $u$ and $v$. Suppose not; then $v$ is contained in every rigid component of~$G$.
    Let $R_1, \ldots, R_\ell$ be the rigid components of $G$.
    Since $G$ is flexible, $\ell\geq 2$.  As noted previously, distinct rigid components can only intersect in at most one vertex, and thus $V(R_i) \cap V(R_j)=\{v\}$ for all $i<j$.  Hence, $v$ is a cut vertex of $G$, which contradicts the $2$-connectivity of $G$.
    \end{proof}

Let us highlight that the previous proof is algorithmic, in the following sense.

\begin{proposition}\label{thm:flexible_stable_poly}
    Given a flexible graph $G$, a stable cut, and hence also a NAC-colouring, of~$G$ can be found in time $O(|V(G)|^3)$.
\end{proposition}

\begin{proof}
    Let $n = |V(G)|$ and $m=|E(G)|$. If $G$ is not connected, then the empty set is a stable cut.
    It is well-known that rigid components can be found in $O(n^2)$ time, see, e.g.~\cite{LeeStreinu}.
    Hence we can also find in $O(n^2)$ time a pair $u,v \in V(G)$ of vertices such that no rigid component contains both $u$ and $v$.
    \Cref{alg:stableCutFlexible} below describes how to find a stable cut of a connected flexible graph~$G$ separating $u$ and $v$.
    It follows the proof of \Cref{prop:flexibleHasStableSeparator}.
    The algorithm is recursive, but there is no branching: each iteration calls the algorithm at most once, on a graph
    that has at least one vertex fewer and at least two edges fewer.
    Hence, the total number of calls is at most $n$.
    Checking whether the neighbourhood of a vertex is stable can be done in $O(m+n^2)$ time.
    The edge contraction takes $O(m)$ operations.
    Checking whether a given vertex is a cut and the construction of the graphs $G_1$ and $G_2$ takes $O(m)$ steps.
    Hence in total, the running time of \Cref{alg:stableCutFlexible} is bounded by $O(n(m+n^2))=O(n^3)$.
\end{proof}

\begin{algorithm}[ht]
    \caption{\textsc{Stable cut of a connected flexible graph}}
    \label{alg:stableCutFlexible}
    \begin{algorithmic}[1]
        \Require a connected flexible graph $G$, vertices $u$ and $v$ not in the same rigid component of $G$
        \Ensure a stable cut $S$ of $G$ such that $u$ and $v$ are separated by $S$
        \If{the neighbourhood of $u$ is stable}
        \State\Return the neighbourhood of $u$
        \Else
        \State $x_1,x_2 :={}$ neighbours of $u$ such that $(u,x_1,x_2)$  is a $3$-cycle
        \For{$i\in\{1,2\}$}
            \State $G'_i :={}$ the graph obtained from $G$ by contracting the edge $ux_i$
            \State $u'_i :={}$ the vertex of $G'_i$ corresponding to the contracted edge $ux_i$
        \EndFor 
        \If{$u'_1$ and $v$ are in different rigid components of $G'_1$}
        \State\Return a stable cut of $G'_1$ separating $u'_1$ and $v$
        \Else
        \State\Return a stable cut of $G'_2$ separating $u'_2$ and $v$
        \EndIf
        \EndIf
    \end{algorithmic}
\end{algorithm}

The algorithm is implemented in the package \textsc{PyRigi}~\cite{pyrigi} as method \texttt{stable\_separating\_set} thanks to Petr Laštovička.
To close this section, we give an alternative proof of \Cref{cor:flexiblestablecut}. While this proof is not as elementary as the one above, it highlights the connection between stable cuts and rigidity on the sphere. We start by briefly explaining what we mean by the latter.

Analogously to realisations in $\RR^2$,
we may consider realisations on the sphere. Let \[\mathcal{S}=\{(x,y,z)\in \RR^3: x^2+y^2+z^2=1\}\] denote the unit sphere in $\RR^3$.
A realisation $(G,p)$ in $\RR^3$ is a \emph{realisation on the sphere} if $p(v)\in \mathcal{S}$ for all $v\in V(G)$. A realisation $(G,p)$ on the sphere is \emph{$\mathcal{S}$-rigid} if the only edge-length-preserving continuous deformations of $(G,p)$ that keep the points on the sphere are the ones arising from isometries of $\mathcal{S}$; otherwise, $(G,p)$ is \emph{$\mathcal{S}$-flexible}. The graph $G$ is \emph{$\mathcal{S}$-rigid} if its generic realisations on the sphere are all $\mathcal{S}$-rigid, and \emph{$\mathcal{S}$-flexible} otherwise. (Here, ``generic'' should be understood as generic among realisations on the sphere; of course, a realisation that is generic in the usual sense cannot lie on 
the sphere.)

The following folklore result links rigidity in the plane to rigidity on the sphere.

\begin{theorem}[See, e.g., \cite{Coning,EJNSTW}]\label{lem:planesphere}
    A graph is rigid if and only if it is $\mathcal{S}$-rigid.    
\end{theorem}

\Cref{thm:NACiffflex} turns out to have a spherical analogue, based on the following notion. Given a graph $G$ and an edge colouring $c \colon E(G) \rightarrow \{ \red, \blue \}$, let us say that a path $(u_1,u_2,u_3,u_4)$ of length $3$ is \emph{alternating} if $c(u_1u_2)\neq c(u_2u_3)\neq c(u_3u_4)$.

\begin{definition}
    Let $G$ be a graph. An edge colouring $c \colon E(G) \rightarrow \{ 
    \red, \blue \}$ is a \emph{NAP-colouring} if it is surjective,
    every $3$-cycle in~$G$
    is monochromatic, and there are no alternating paths of length $3$ in~$G$. 
\end{definition}
NAP stands for ``No Alternating Path''.
Notice that a NAP-colouring is also a NAC-colouring. Also note that a surjective edge colouring with two colours is a NAP-colouring if and only if every edge $uv$ has an endvertex $v$ such that all edges incident to $v$ have the same colour. 

\begin{theorem}[{\cite[Theorem 3.14]{GGLS2021}}]
    \label{thm:NAP}
    A connected graph has a quasi-injective $\mathcal{S}$-flexible realisation if and only if it has a NAP-colouring.
\end{theorem}

To complete the picture, we observe that there is a one-to-one correspondence between stable separations and NAP-colourings, up to swapping colours.
The first implication of the following lemma is a more precise form of \Cref{thm:stableCutImpliesNAC}.

\begin{lemma}
    \label{lem:NAPiffStableCut}
    Let $G$ be a graph without isolated vertices. If $\{G_1,G_2\}$ is a stable separation of~$G$, then colouring all edges in $G_1$ red and all edges in $G_2$ blue yields a NAP-colouring of $G$.  Conversely, if $c:E(G) \to \{\red, \blue\}$ is a NAP-colouring of $G$, then $\{G_{\red}, G_{\blue}\}$ is a stable separation of $G$, where $G_{\red}$ and $G_{\blue}$ are the subgraphs of $G$ induced by the red and blue edges, respectively.  
\end{lemma}

\begin{proof}
    First, suppose $\{G_1,G_2\}$ is a stable separation of $G$, and for each $e \in E(G)$ let us define a colouring $c \colon E(G) \rightarrow \{ \red, \blue \}$ by letting $c(e)=\red$ if $e \in E(G_1)$ and $c(e) = \blue$ if $e \in E(G_2)$. Note that since $\{G_1,G_2\}$ is a stable separation, $E(G_1)\cap E(G_2)=\varnothing$, so $c$ is well-defined (up to swapping colours). Also, since $E(G_1)$ and $E(G_2)$ are both nonempty, $c$ is surjective. Let $uv \in E(G)$.  Since $V(G_1) \cap V(G_2)$ is a stable set, at least one of $u$ or $v$ (say $u$) is not in $V(G_1) \cap V(G_2)$.  Thus, all edges incident to $u$ are the same colour.  It follows that $c$ is a NAP-colouring.  
    
    Conversely, suppose $c:E(G) \to \{\red, \blue\}$ is a NAP-colouring of $G$, and let $G_{\red}$ and $G_{\blue}$ be the subgraphs of $G$ induced by the red and blue edges, respectively.    We first show that $V(G_{\red}) \cap V(G_{\blue})$ is a stable set of $G$.  Suppose not, and let $\{u,v\} \subseteq V(G_{\red}) \cap V(G_{\blue})$ be an adjacent pair of vertices.  By symmetry, we may assume that $uv$ is red. Since $\{u,v\} \subseteq V(G_{\blue})$, there are blue edges meeting $u$ and $v$, which together with $uv$ contradict $c$ being a NAP-colouring.
    Therefore,  $V(G_{\red}) \cap V(G_{\blue})$ is a stable set of $G$, as claimed.   
    It remains to show that $V(G_{\red}) \setminus V(G_{\blue})$ and $V(G_{\blue}) \setminus V(G_{\red})$ are both nonempty.  Towards a contradiction, suppose, without loss of generality, that $V(G_{\red}) \subseteq V(G_{\blue})$.  But then $V(G_{\red}) \subseteq V(G_{\blue})\cap V(G_{\red})$, which is stable, so there are no red edges, contradicting the surjectivity of $c$.
\end{proof}

\begin{proof}[Alternative proof of \Cref{cor:flexiblestablecut}]
    Let $G$ be a flexible graph.  By \Cref{lem:planesphere}, $G$ is $\mathcal{S}$-flexible.
    Hence, $G$ has a NAP-colouring $c$ by \Cref{thm:NAP}, which gives a stable cut by \Cref{lem:NAPiffStableCut}.
\end{proof}

\section{The number of NAC-colourings}\label{sec:numberofNACs}

Recall that for a graph $G$, we let $\nnac{G}$ denote the number of NAC-colourings of $G$ divided by two. By \Cref{prop:motionbound}, this number gives a lower bound on the number of angle-different motions of $G$. Thus, we may view $\nnac{G}$ as a combinatorial measure of the ``non-generic flexibility'' of a graph. 
From this viewpoint, there are two natural extremal problems: one is to maximise the number of NAC-colourings in rigid graphs, and the other is to minimise the number of NAC-colourings in flexible graphs. 

In this section, we investigate both of these problems. First we show that the number of NAC-colourings of an $n$-vertex graph is $o(4^n)$. On the other hand, we construct minimally rigid graphs on $n$ vertices having over $2^n$ NAC-colourings (\Cref{cor:2+}). For the second problem, we show that the flexible graphs with the minimum number of NAC-colourings are precisely those that arise by gluing a pair of graphs $G_1,G_2$ satisfying $\nnac{G_1}=\nnac{G_2} = 0$ along a vertex (\Cref{prop:flexibleOneNacChar}).
 
\subsection{A bound for dense graphs}

We start by giving a general upper bound on the number of NAC-colourings. While this is only a slight improvement on $2^{|E(G)|}$ in the minimally rigid case, it is substantially better for dense graphs.

\begin{theorem}\label{thm:NACupperbound}
    For every graph $G$ on $n$ vertices, $\nnac{G}\leq \frac{1}{2}\binom{2n-4}{n-2}\sim \frac{2^{2n-5}}{\sqrt{\pi n}}$.
\end{theorem}

\noindent The asymptotic bound follows from Stirling's formula; hence we only prove the inequality.

\begin{proof}
    Let us define a \textit{partial NAC-colouring} of $G$ to be a red-blue colouring of some subset of $E(G)$ with the property that
    there is no almost red or almost blue cycle and there is no entirely red or entirely blue spanning forest with the same number of components as $G$.
    Note that a NAC-colouring is precisely a partial NAC-colouring in which every edge is coloured. Moreover, restricting a NAC-colouring to a subset of the edges gives a partial NAC-colouring. 
    However, a partial NAC-colouring cannot necessarily be extended to a NAC-colouring.
    
    For the rest of the proof, fix an ordering $e_1,\ldots, e_m$ of the edges of $G$.
    We build up a sequence of partial NAC-colourings $c_0, \ldots, c_m$, where $c_0$ is the empty colouring, using the following procedure. 
    Choose a sequence $s: [m] \to \{\red,\blue\}$, and set $j=1$.
    Now take the edges in order. If there is exactly one way to extend $c_{i-1}$ to a partial NAC-colouring $c_i$ by colouring $e_i$, do this. If there are no ways, stop the process. If there are two ways, colour $e_i$ with the colour $s(j)$ and increase $j$ by one.
    
    First, we claim that every NAC-colouring $c$ of $G$ can be obtained using this procedure for some $s: [m] \to \{\red,\blue\}$. Indeed, for each $i \in [m]$,
    consider whether the colouring obtained from $c$ by restricting it to $\{e_1,\ldots, e_i\}$ and swapping the colour of $e_i$ is a partial NAC-colouring;
    let $i_1<\cdots <i_k$ be the set of indices~$i$ for which it is. It is not difficult to see that if the sequence $s$ in the procedure starts with $c(e_{i_1}),\ldots,c(e_{i_k})$, then we recover $c$. 
    
    Hence, we only need to show that we can construct at most $\binom{2n-4}{n-2}$ different NAC-colourings by using the above procedure. (Here we are not factoring out by swapping colours, so we have to divide by $2$ in the end.)
    Note that if at step $i$ we use a colour from $s$, say red, then the endvertices of $e_i$ must be in different red components in $c_{i-1}$ (since it was possible to colour $e_i$ blue), and so there are fewer red components in $c_i$ than $c_{i-1}$. Here, by a red component in $c_i$, we refer to the components of the graph on vertex set $V(G)$ whose edges are the ones coloured red by $c_i$. Since $c_i$ is a partial NAC-colouring, there are at least $2$ red components in $c_i$. This means that during the procedure, we can read off the value $\red$ from $s$ at most $n-2$ times, and likewise for the value $\blue$.
    
    It follows that instead of sequences of length $m$, it suffices to consider sequences of length $2n-4$ in which the value $\red$ appears exactly $n-2$ times. The number of such sequences is~$\binom{2n-4}{n-2}$, and thus the number of NAC-colourings that can be obtained using the above procedure is at most this number. 
\end{proof}

We do not know how close \Cref{thm:NACupperbound} is to being tight. It is easy to find sparse graphs on $n$ vertices with roughly $2^n$ NAC-colourings: for example, $\nnac{G}=2^{n-2}-1$ for every $n$-vertex tree, and $\nnac{G}=2^{n-1}-(n+1)$ if $G$ is the $n$-vertex cycle. Our next result shows that dense graphs can also have exponentially many NAC-colourings.

\begin{proposition}\label{prop:NAC_Kmn}
    For every pair of positive integers $n_1,n_2$, $\nnac{K_{n_1, n_2}}=2^{n_1+n_2-2}-1$. 
\end{proposition}

\begin{proof}
    Let $(X,Y)$ be the bipartition of $G$ with $|X|=n_1$ and $|Y|=n_2$.  For each $X' \subseteq X$ and $Y' \subseteq Y$ such that $(X',Y') \notin \{(\varnothing, \varnothing), (X, \varnothing), (\varnothing, Y), (X,Y)\}$, we let $c_{X',Y'}$ be the colouring of $E(K_{n_1, n_2})$ such that all edges between~$X'$ and $Y'$ are blue, all edges between $X \setminus X'$ and $Y \setminus Y'$ are blue, all edges between~$X'$ and $Y \setminus Y'$ are red, and all edges between $X \setminus X'$ and $Y'$ are red.  Note that $c_{X',Y'}$ is a NAC-colouring. Moreover, $c_{X_1', Y_2'}=c_{X_1, Y_2}$ (up to flipping colours) if and only if $(X_1', Y_1') \in \{(X_1, Y_1), (X_1, Y \setminus Y_1), (X \setminus X_1, Y_1), (X \setminus X_1, Y \setminus Y_1)\}$. Thus,
\[
    \nnac{G} \geq \frac{1}{4} (2^{n_1} 2^{n_2}-4)=2^{n_1+n_2-2}-1.
\]
    To prove equality, we show that every NAC-colouring of $G$ is equal to $c_{X',Y'}$ for some $X' \subseteq X$ and $Y' \subseteq Y$.  Let $c$ be a NAC-colouring of $G$.
    
    \begin{claim}
        Let $B$ be a blue component of $c$ and $R$ be a red component of $c$ such that $V(B) \cap V(R) \cap Y \neq \varnothing$.  Then
        $V(B) \cap Y=V(R)\cap Y$ and $V(B) \cap X= X \setminus V(R)$.  
    \end{claim}
    
    \begin{proof}[Proof of Claim]
        Let $X_B:=V(B) \cap X$, $Y_B:=V(B) \cap Y$, $X_R:=V(R) \cap X$ and $Y_R:=V(R) \cap Y$.  By assumption $Y_B \cap Y_R \neq \varnothing$.  We claim that $X_B \cap X_R =\varnothing$.  Suppose not, and let $x \in X_B \cap X_R $ and $y \in Y_B \cap Y_R$.  Since $B$ is an induced subgraph (recall the equivalent conditions below \Cref{def:NAC}), $xy$ is blue.  On the other hand, since $R$ is an induced subgraph, $xy$ is red, which is a contradiction.  
        Suppose $Y_R \setminus Y_B \neq \varnothing$.  Choose $x_B \in X_B$ and $y_R \in Y_R \setminus Y_B$.  Then $x_By_R$ cannot be blue by the maximality of $B$.  On the other hand, $x_By_R$ cannot be red since $X_B \cap X_R=\varnothing$.  Thus, $Y_R \subseteq Y_B$. Analogously by swapping $R$ and $B$, we have $Y_B \setminus Y_R = \varnothing$. Thus, $Y_B=Y_R$.
        It remains to show that $X_R=X \setminus X_B$.  Since $X_B \cap X_R=\varnothing$, it suffices to show that $X \setminus (X_B \cup X_R) = \varnothing$.  If $x \in X \setminus (X_B \cup X_R)$ and $y \in Y_B=Y_R$, then $xy$ cannot be red by the maximality of $R$ and $xy$ cannot be blue by the maximality of $B$.  Thus,  $X \setminus (X_B \cup X_R) \neq \varnothing$, as required.  
    \end{proof}
    
    Let $R$ be a red component and $B$ a blue component such that $V(B) \cap V(R) \cap Y \neq \varnothing$.  By the above claim, $V(B) \cap Y=V(R)\cap Y$ and $V(B) \cap X= X \setminus V(R)$.  Let $B'$ be a blue component such that  $V(B') \cap V(R) \cap X \neq \varnothing$.  By the above claim, $V(B') \cap X=V(R) \cap X$ and $V(B') \cap Y=Y \setminus V(R)$.  Let $R'$ be a red component such that $V(R') \cap V(B') \cap Y \neq \varnothing$.  By the above claim, $V(R') \cap Y=V(B') \cap Y$, and $V(R') \cap X=X \setminus V(B')$.  Thus, setting $X'=V(B) \cap X$ and $Y'=V(B) \cap Y$, we have $c=c_{X', Y'}$, as required. 
\end{proof}

Given two graphs $G_1$ and $G_2$, we let $G_1 \sqcup G_2$ denote the disjoint union of $G_1$ and $G_2$.  The following proposition gives an explicit formula for $\nnac{G_1 \sqcup G_2}$.  We omit the proof since it is a special case of a more general formula (\Cref{lem:blockproduct}), which we will prove later.  

\begin{proposition} \label{prop:disjointunion}
    For all graphs $G_1$ and $G_2$,
    \begin{equation*}
        \nnac{G_1 \sqcup G_2}=
        \begin{cases}
            \nnac{G_1}, & \text{if $E(G_2)=\varnothing$}, \\
            \nnac{G_2}, & \text{if $E(G_1)=\varnothing$}, \\
            2(\nnac{G_1}+1)(\nnac{G_2}+1)-1,  & \text{if $E(G_1) \neq \varnothing$ and $E(G_2) \neq \varnothing$.}
        \end{cases}
    \end{equation*}
\end{proposition}

The \emph{join} of $G_1$ and $G_2$, denoted by $G_1 \vee G_2$, is the graph obtained from $G_1 \sqcup G_2$ by adding all edges between $V(G_1)$ and $V(G_2)$. For a graph $G$, we let $\lambda(G)$ be the number of connected components of $G$.
Note that if $H$ is a connected component of $G_1$ with $|V(H)| \geq 2$, then in every NAC-colouring of $G_1 \vee G_2$, $H$ is monochromatic and every edge between $V(G_2)$ and $V(H)$ must be the same colour as $H$.  Combining this observation with~\Cref{prop:NAC_Kmn} gives the following explicit formula for $\nnac{G_1 \vee G_2}$.

\begin{proposition} \label{lem:join}
    For all graphs $G_1$ and $G_2$,
    \begin{equation*}
        \nnac{G_1 \vee G_2}=
        \begin{cases}
            0, & \text{if $E(G_1) \neq \varnothing$ and $E(G_2) \neq \varnothing$}, \\
            2^{\lambda(G_1)-1}-1,  & \text{if $E(G_1) \neq \varnothing$ and $E(G_2)=\varnothing$}, \\
            2^{\lambda(G_2)-1}-1,  & \text{if $E(G_2) \neq \varnothing$ and $E(G_1)=\varnothing$}, \\
            2^{\lambda(G_1)+\lambda(G_2)-2}-1, & \text{if $E(G_1)=E(G_2)=\varnothing$.}
        \end{cases}
    \end{equation*}
\end{proposition}

A \emph{cograph} is a graph which can be recursively built using $\sqcup$ and $\vee$ (starting from single vertex graphs).    
Observe that by combining~\Cref{prop:disjointunion} and~\Cref{lem:join} we can efficiently compute $\nnac{G}$ for every cograph $G$.

\subsection{Minimally rigid graphs}

In this subsection, we concentrate on the problem of maximising the number of NAC-colourings in minimally rigid graphs. While we cannot match the upper bound given by \Cref{thm:NACupperbound}, we show that there exists an infinite sequence of minimally rigid graphs $G_k$ for which $\nnac{G_k}$ is exponential in $|V(G_k)|$.

We give two proofs of this fact. The first is a simple construction for which the NAC-colourings can be precisely described, and for which there are approximately $2^{|G_k|}$ of them. The second uses detailed calculations for a small graph to define a sequence for which the base of the exponential is strictly greater than $2$.

See \Cref{fig:NACnumbers_all} and \Cref{fig:NACnumbers_all_11} for the distribution of $\nnac{G}$ for minimally rigid graphs on up to ten vertices and on eleven and twelve vertices respectively, and \Cref{fig:maxNACs} for some minimally rigid graphs that have the maximum number of NAC-colourings on a given number of vertices.
\begin{figure}[t]
    \centering
    \begin{tikzpicture}[yscale=0.5]
        \foreach \xstep in {0.045}
        {
            \foreach \n/\col [count=\i] in {0/black!75, 255/black!75, 31/colB, 63/colR, 127/colG, 307/colY}
            {
                \draw[black!25] ({\xstep*(\n+0.5)},-0.1) node[below,alabelsty] {\textcolor{\col}{$\n$}};
            }
            \draw[black!25] (0,0)--(127*\xstep,0);
            \foreach \y/\col [count=\i] in {25/colB, 167/colR, 1410/colG, 12219/colY, 1/black!75}
            {
                \draw[black!25] (0,{log2(1+\y)})--(-0.1,{log2(1+\y)}) node[left,alabelsty] {\textcolor{\col}{$\y$}};
                \draw[black!25] (0,{log2(1+\y)})--(308*\xstep,{log2(1+\y)});
            }
            \node[rotate=-90,alabelsty] at (-1,{0.5*log2(12219)}) {\# graphs (log)};
            \node[alabelsty] at (153*\xstep,-1.5) {$\nnac{G}$};
            
            \foreach \nacs/\graphs [count=\i] in {0/529, 1/4159, 2/2478, 3/11432, 4/4722, 5/4359, 6/6372, 7/12219, 8/5221, 9/4269, 10/3035, 11/3179, 12/4390, 13/4943, 14/2117, 15/7609, 16/1610, 17/2099, 18/1877, 19/1362, 20/1106, 21/1336, 22/752, 23/1461, 24/938, 25/1934, 26/533, 27/1793, 28/296, 29/436, 30/545, 31/2417, 32/410, 33/300, 34/433, 35/379, 36/205, 37/404, 38/173, 39/240, 40/182, 41/223, 42/84, 43/277, 44/149, 45/286, 46/157, 47/327, 48/66, 49/224, 50/127, 51/446, 52/48, 53/79, 54/114, 55/272, 56/93, 57/45, 58/61, 59/66, 60/48, 61/111, 62/39, 63/669, 64/27, 65/74, 66/47, 67/50, 68/22, 69/72, 70/25, 71/40, 72/31, 73/32, 74/12, 75/60, 76/19, 77/22, 78/24, 79/37, 80/13, 81/22, 82/20, 83/25, 84/12, 85/36, 86/30, 87/54, 88/8, 89/38, 90/15, 91/70, 93/29, 94/19, 95/58, 96/3, 97/16, 98/12, 99/47, 100/18, 101/5, 102/9, 103/85, 104/11, 105/9, 106/6, 107/9, 108/12, 109/22, 110/3, 111/20, 112/3, 113/13, 114/4, 115/4, 116/3, 117/3, 118/8, 119/3, 120/2, 121/3, 122/3, 123/8, 124/4, 125/5, 126/5, 127/138, 128/2, 129/4, 130/2, 131/4, 133/4, 134/4, 135/2, 136/4, 137/1, 138/2, 139/1, 140/1, 141/1, 142/3, 144/1, 145/6, 146/2, 147/1, 148/1, 149/1, 150/5, 151/1, 152/3, 153/1, 154/6, 155/1, 156/1, 157/5, 158/1, 159/2, 160/4, 161/2, 162/4, 163/2, 164/3, 165/6, 166/1, 167/1, 168/1, 170/1, 171/3, 172/1, 173/9, 174/2, 175/2, 176/1, 179/1, 180/2, 181/5, 182/1, 183/1, 184/1, 185/1, 186/3, 187/9, 188/2, 190/1, 191/1, 192/2, 194/2, 197/2, 198/2, 199/1, 201/5, 202/2, 203/1, 204/1, 205/2, 208/1, 209/4, 213/3, 217/2, 219/8, 221/1, 222/2, 225/3, 227/6, 231/1, 234/1, 235/1, 237/1, 240/1, 241/1, 247/4, 255/120, 257/1, 301/1, 302/1, 307/1}
            {
                \draw[chart,draw=colY,fill=colY!50!white] (\xstep*\nacs,0) rectangle (\xstep*\nacs+\xstep,{log2(1+\graphs)});
            }
            
            \foreach \nacs/\graphs [count=\i] in {0/136, 1/742, 2/332, 3/1410, 4/450, 5/304, 6/547, 7/976, 8/302, 9/169, 10/143, 11/106, 12/245, 13/209, 14/61, 15/379, 16/37, 17/36, 18/74, 19/19, 20/21, 21/19, 22/25, 23/36, 24/23, 25/65, 26/6, 27/42, 28/3, 29/4, 30/19, 31/105, 32/10, 33/5, 34/10, 35/5, 36/2, 37/9, 38/1, 39/4, 40/1, 41/4, 42/1, 43/6, 44/3, 45/9, 46/4, 47/9, 48/1, 49/7, 50/1, 51/14, 52/1, 53/1, 54/6, 55/2, 58/1, 62/1, 63/21, 66/1, 72/2, 78/1, 82/1, 85/1, 86/2, 87/1, 90/1, 93/2, 98/1, 100/1, 104/1, 109/2, 113/1, 123/1, 127/19}
            {
                \draw[chart,draw=colG,fill=colG!50!white] (\xstep*\nacs,0) rectangle (\xstep*\nacs+\xstep,{log2(1+\graphs)});
            }
            \foreach \nacs/\graphs [count=\i] in {0/39, 1/132, 2/39, 3/167, 4/34, 5/14, 6/37, 7/67, 8/8, 9/4, 10/3, 11/1, 12/13, 13/6, 14/1, 15/22, 16/2, 18/3, 22/1, 23/1, 24/1, 25/3, 31/3, 46/1, 54/1, 63/5}
            {
                \draw[chart,draw=colR,fill=colR!50!white] (\xstep*\nacs,0) rectangle (\xstep*\nacs+\xstep,{log2(1+\graphs)});
            }
            \foreach \nacs/\graphs [count=\i] in {0/12, 1/25, 2/4, 3/18, 4/1, 6/2, 7/5, 12/1, 15/1, 31/1}
            {
                \draw[chart] (\xstep*\nacs,0) rectangle (\xstep*\nacs+\xstep,{log2(1+\graphs)});
            }
        }
    \end{tikzpicture}
    \caption{The numbers of minimally rigid graphs on 7 (blue), 8 (red), 9 (green) and 10 vertices (yellow)
    according to the number of NAC-colourings.
    The maximal values for each number of vertices are indicated by the labels.
    On 6 vertices, there are 5 graphs with no NAC-colouring, 5 with $\nnac{G}=1$, 3 with $\nnac{G}=2$,
    and the complete bipartite graph $K_{3,3}$ with $\nnac{K_{3,3}}=15$.
    The numbers were determined computationally using the \textsc{SageMath} package \textsc{FlexRiLoG}~\cite{FlexRiLoGPaper} and the database~\cite{minRigidDatabase}.
    }
    \label{fig:NACnumbers_all}
\end{figure}
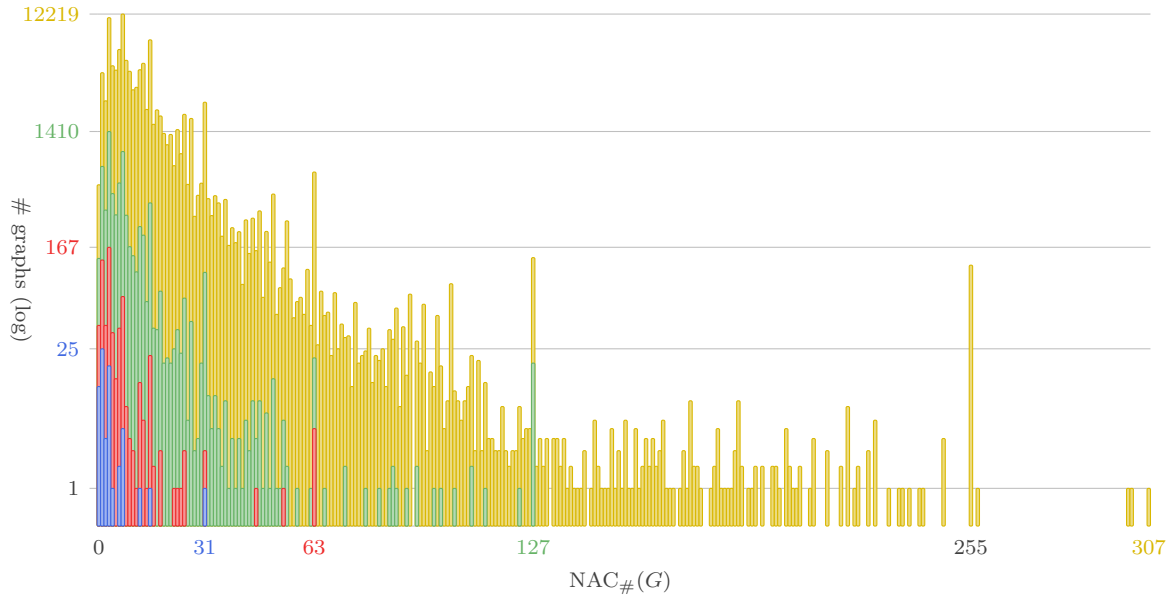

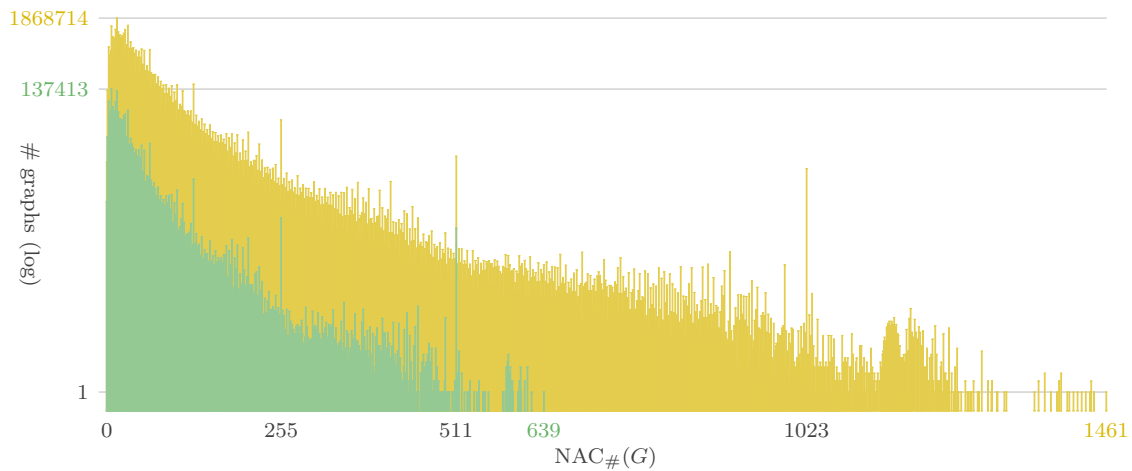
\begin{figure}[ht]
    \centering
    \begin{tikzpicture}[yscale=0.25]
        \foreach \xstep in {0.009}
        {
            \foreach \n in {0, 255, 511, 1023}
            {
                \draw[black!25] ({\xstep*(\n+0.5)},-0.1) node[below,alabelsty] {\textcolor{black!75}{$\n$}};
            }
            \foreach \n/\col [count=\i] in {639/colG, 1461/colY}
            {
                \draw[black!25] ({\xstep*(\n+0.5)},-0.1) node[below,alabelsty] {\textcolor{\col}{$\n$}};
            }
            \draw[black!25] (0,0)--(127*\xstep,0);
            \foreach \y/\logy/\col [count=\i] in {1/1/black!75, 137413/17.0681/colG, 1868714/20.8336/colY}
            {
                \draw[black!25] (0,\logy)--(-0.1,\logy) node[left,alabelsty] {\textcolor{\col}{$\y$}};
                \draw[black!25] (0,\logy)--(1462*\xstep,\logy);
            }
            \node[rotate=-90,alabelsty] at (-1,{0.5*20.8336}) {\# graphs (log)};
            \node[alabelsty] at (730*\xstep,-2.5) {$\nnac{G}$};
            
            \foreach \nacs/\graphs [count=\i] in {0/13.19, 1/17.01, 2/16.81, 3/19.30, 4/18.41, 5/18.87, 6/19.09, 7/20.41, 8/19.39, 9/19.85, 10/19.08, 11/19.77, 12/19.50, 13/20.24, 14/19.20, 15/20.83, 16/19.27, 17/20.10, 18/19.31, 19/19.95, 20/19.17, 21/19.82, 22/18.84, 23/19.91, 24/19.07, 25/20.03, 26/18.90, 27/20.23, 28/18.64, 29/19.43, 30/18.74, 31/20.45, 32/18.71, 33/19.23, 34/18.69, 35/19.57, 36/18.28, 37/19.21, 38/18.24, 39/19.26, 40/18.33, 41/18.97, 42/17.93, 43/19.13, 44/17.96, 45/18.76, 46/18.02, 47/19.07, 48/17.77, 49/18.75, 50/17.80, 51/19.19, 52/17.47, 53/18.38, 54/17.61, 55/19.12, 56/17.42, 57/18.05, 58/17.33, 59/18.30, 60/17.19, 61/18.01, 62/17.07, 63/19.17, 64/16.88, 65/17.95, 66/16.94, 67/17.85, 68/16.79, 69/17.80, 70/16.78, 71/17.85, 72/16.56, 73/17.34, 74/16.55, 75/17.72, 76/16.38, 77/17.21, 78/16.35, 79/17.48, 80/16.17, 81/17.09, 82/16.11, 83/17.28, 84/16.17, 85/16.83, 86/16.04, 87/17.31, 88/15.82, 89/16.84, 90/16.06, 91/17.11, 92/15.68, 93/16.68, 94/15.80, 95/17.25, 96/15.46, 97/16.34, 98/15.60, 99/16.91, 100/15.32, 101/16.45, 102/15.45, 103/17.29, 104/15.19, 105/15.98, 106/15.22, 107/16.26, 108/15.31, 109/16.16, 110/14.95, 111/16.97, 112/14.92, 113/15.92, 114/14.83, 115/15.86, 116/14.86, 117/15.69, 118/14.89, 119/15.92, 120/14.67, 121/15.66, 122/14.58, 123/15.90, 124/14.57, 125/15.31, 126/14.48, 127/17.33, 128/14.36, 129/15.20, 130/14.50, 131/15.68, 132/14.19, 133/15.10, 134/14.23, 135/15.36, 136/14.15, 137/14.92, 138/13.91, 139/15.48, 140/14.01, 141/14.81, 142/14.08, 143/15.21, 144/13.92, 145/14.70, 146/13.94, 147/14.92, 148/13.87, 149/14.60, 150/13.78, 151/15.19, 152/13.58, 153/14.45, 154/13.63, 155/14.81, 156/13.46, 157/14.27, 158/13.50, 159/14.76, 160/13.52, 161/14.26, 162/13.42, 163/14.60, 164/13.32, 165/14.20, 166/13.27, 167/14.63, 168/13.21, 169/14.23, 170/13.08, 171/14.41, 172/12.94, 173/14.12, 174/12.98, 175/14.70, 176/12.81, 177/13.77, 178/12.98, 179/14.48, 180/12.80, 181/13.98, 182/12.62, 183/14.70, 184/12.61, 185/13.51, 186/12.68, 187/14.21, 188/12.70, 189/13.67, 190/12.61, 191/14.64, 192/12.54, 193/13.22, 194/12.45, 195/13.67, 196/12.46, 197/13.28, 198/12.47, 199/14.37, 200/12.37, 201/13.26, 202/12.14, 203/13.78, 204/12.49, 205/13.29, 206/12.09, 207/14.78, 208/12.05, 209/12.96, 210/11.95, 211/13.22, 212/12.21, 213/13.01, 214/12.04, 215/13.29, 216/12.16, 217/13.26, 218/11.84, 219/13.53, 220/11.90, 221/12.76, 222/11.79, 223/14.15, 224/11.69, 225/12.70, 226/11.68, 227/13.30, 228/11.73, 229/12.36, 230/11.71, 231/12.88, 232/11.57, 233/12.58, 234/11.61, 235/12.92, 236/11.49, 237/12.55, 238/11.38, 239/12.81, 240/11.45, 241/12.27, 242/11.30, 243/12.87, 244/11.28, 245/12.14, 246/11.29, 247/13.07, 248/11.10, 249/12.05, 250/11.15, 251/12.34, 252/11.22, 253/11.93, 254/11.12, 255/15.43, 256/11.04, 257/12.00, 258/10.98, 259/12.33, 260/10.83, 261/11.83, 262/10.89, 263/12.61, 264/10.71, 265/11.62, 266/10.80, 267/12.15, 268/10.80, 269/11.71, 270/10.86, 271/12.24, 272/10.74, 273/11.60, 274/10.74, 275/11.82, 276/10.67, 277/11.47, 278/10.71, 279/12.47, 280/10.71, 281/11.41, 282/10.79, 283/11.81, 284/10.59, 285/11.56, 286/10.62, 287/12.00, 288/10.48, 289/11.24, 290/10.46, 291/11.81, 292/10.29, 293/11.37, 294/10.48, 295/11.80, 296/10.01, 297/11.30, 298/10.14, 299/11.57, 300/10.10, 301/11.32, 302/10.28, 303/12.16, 304/10.24, 305/11.11, 306/10.33, 307/11.43, 308/9.920, 309/11.16, 310/10.26, 311/11.78, 312/10.05, 313/10.91, 314/9.977, 315/11.34, 316/10.07, 317/10.92, 318/9.809, 319/11.80, 320/9.852, 321/11.04, 322/9.822, 323/11.33, 324/9.908, 325/10.91, 326/9.738, 327/11.59, 328/9.662, 329/10.90, 330/9.596, 331/11.45, 332/9.706, 333/10.71, 334/9.604, 335/11.59, 336/9.527, 337/10.67, 338/9.683, 339/11.40, 340/9.564, 341/10.43, 342/9.457, 343/11.59, 344/9.308, 345/10.31, 346/9.655, 347/11.56, 348/9.468, 349/10.34, 350/9.496, 351/11.86, 352/9.433, 353/10.35, 354/9.290, 355/10.78, 356/9.655, 357/10.62, 358/9.628, 359/11.75, 360/9.564, 361/10.42, 362/9.190, 363/11.25, 364/9.297, 365/10.04, 366/9.111, 367/12.12, 368/8.983, 369/10.18, 370/9.177, 371/10.37, 372/9.401, 373/10.50, 374/9.116, 375/11.29, 376/9.124, 377/10.45, 378/8.842, 379/11.06, 380/8.867, 381/9.905, 382/8.980, 383/11.78, 384/9.066, 385/10.14, 386/8.577, 387/10.20, 388/8.898, 389/9.945, 390/8.980, 391/10.73, 392/8.577, 393/9.963, 394/8.718, 395/10.29, 396/8.611, 397/10.33, 398/8.669, 399/11.70, 400/8.435, 401/9.687, 402/8.371, 403/10.77, 404/8.276, 405/9.539, 406/8.855, 407/10.33, 408/8.889, 409/10.19, 410/8.362, 411/10.66, 412/8.480, 413/9.451, 414/8.520, 415/12.17, 416/8.371, 417/9.439, 418/8.488, 419/10.02, 420/8.238, 421/9.324, 422/8.644, 423/9.966, 424/8.335, 425/9.326, 426/8.379, 427/10.15, 428/8.349, 429/9.589, 430/8.388, 431/9.932, 432/8.000, 433/9.459, 434/8.103, 435/10.60, 436/8.082, 437/9.011, 438/8.349, 439/10.46, 440/7.755, 441/8.748, 442/7.937, 443/9.644, 444/8.109, 445/8.997, 446/7.814, 447/10.83, 448/7.960, 449/9.033, 450/7.925, 451/9.665, 452/7.468, 453/8.633, 454/7.555, 455/10.23, 456/7.366, 457/8.807, 458/7.401, 459/8.940, 460/7.794, 461/8.721, 462/7.644, 463/9.373, 464/7.607, 465/8.607, 466/7.644, 467/8.969, 468/7.672, 469/8.931, 470/7.331, 471/9.299, 472/7.150, 473/8.271, 474/7.883, 475/9.537, 476/7.219, 477/8.140, 478/7.539, 479/8.830, 480/7.401, 481/8.581, 482/7.500, 483/8.697, 484/7.180, 485/7.954, 486/6.858, 487/8.910, 488/7.022, 489/8.050, 490/7.044, 491/8.562, 492/6.943, 493/8.335, 494/7.435, 495/9.437, 496/7.276, 497/7.775, 498/7.257, 499/8.629, 500/6.989, 501/8.061, 502/6.907, 503/8.570, 504/6.858, 505/8.011, 506/6.672, 507/8.358, 508/7.066, 509/8.082, 510/7.066, 511/13.51, 512/6.781, 513/7.794, 514/7.295, 515/8.615, 516/6.714, 517/7.852, 518/6.807, 519/8.585, 520/6.883, 521/7.768, 522/6.919, 523/7.960, 524/6.820, 525/7.775, 526/6.919, 527/8.508, 528/6.755, 529/7.539, 530/6.954, 531/7.781, 532/6.807, 533/7.150, 534/6.907, 535/8.492, 536/6.426, 537/7.418, 538/6.629, 539/7.972, 540/6.883, 541/7.735, 542/6.524, 543/8.124, 544/6.768, 545/7.229, 546/6.728, 547/7.827, 548/6.644, 549/7.418, 550/6.768, 551/7.983, 552/6.555, 553/7.459, 554/6.700, 555/7.925, 556/6.615, 557/7.839, 558/6.781, 559/7.966, 560/6.700, 561/7.820, 562/6.508, 563/7.794, 564/6.570, 565/7.577, 566/6.190, 567/8.006, 568/6.190, 569/7.629, 570/6.022, 571/7.943, 572/6.170, 573/7.665, 574/6.508, 575/7.392, 576/6.267, 577/7.257, 578/5.977, 579/7.665, 580/6.190, 581/7.600, 582/6.459, 583/8.285, 584/6.570, 585/7.547, 586/7.066, 587/8.243, 588/6.700, 589/7.748, 590/6.954, 591/7.392, 592/6.322, 593/7.451, 594/6.209, 595/7.539, 596/6.285, 597/6.794, 598/6.476, 599/7.409, 600/6.087, 601/7.276, 602/6.285, 603/8.234, 604/6.066, 605/7.672, 606/6.459, 607/7.807, 608/6.022, 609/7.150, 610/5.883, 611/7.983, 612/6.022, 613/7.547, 614/5.833, 615/7.409, 616/5.883, 617/6.870, 618/6.066, 619/8.170, 620/6.000, 621/7.055, 622/5.833, 623/7.476, 624/6.375, 625/7.109, 626/5.524, 627/7.340, 628/5.833, 629/6.833, 630/6.229, 631/7.858, 632/6.066, 633/6.870, 634/5.858, 635/7.500, 636/6.066, 637/7.257, 638/6.150, 639/7.741, 640/6.022, 641/6.870, 642/5.833, 643/7.895, 644/6.000, 645/6.728, 646/5.954, 647/7.375, 648/6.044, 649/6.989, 650/5.833, 651/7.665, 652/6.022, 653/6.845, 654/5.807, 655/7.629, 656/5.977, 657/6.358, 658/5.615, 659/7.592, 660/6.209, 661/6.585, 662/6.340, 663/8.304, 664/6.087, 665/6.768, 666/6.087, 667/7.077, 668/6.000, 669/6.728, 670/5.931, 671/7.728, 672/6.150, 673/6.209, 674/5.755, 675/6.820, 676/6.190, 677/6.508, 678/6.150, 679/6.989, 680/6.066, 681/6.340, 682/5.426, 683/7.304, 684/5.728, 685/6.304, 686/5.322, 687/7.492, 688/5.781, 689/6.044, 690/5.644, 691/6.644, 692/5.883, 693/6.755, 694/4.858, 695/8.418, 696/5.755, 697/6.600, 698/5.555, 699/6.728, 700/5.833, 701/6.820, 702/5.426, 703/7.267, 704/5.672, 705/6.375, 706/5.358, 707/7.180, 708/5.615, 709/6.358, 710/5.492, 711/6.644, 712/5.524, 713/6.989, 714/5.426, 715/6.322, 716/5.248, 717/7.150, 718/5.700, 719/7.375, 720/5.129, 721/7.190, 722/5.392, 723/6.728, 724/5.044, 725/6.409, 726/5.285, 727/8.028, 728/5.392, 729/6.931, 730/4.700, 731/6.570, 732/4.807, 733/6.000, 734/4.644, 735/6.658, 736/4.858, 737/6.629, 738/5.358, 739/6.555, 740/5.000, 741/6.267, 742/5.322, 743/6.833, 744/5.170, 745/6.700, 746/5.209, 747/7.592, 748/5.000, 749/6.229, 750/5.129, 751/8.119, 752/5.129, 753/6.229, 754/5.087, 755/7.129, 756/4.858, 757/6.322, 758/5.392, 759/6.508, 760/4.700, 761/6.000, 762/5.087, 763/7.322, 764/4.087, 765/5.358, 766/5.087, 767/6.267, 768/4.755, 769/6.409, 770/5.129, 771/7.524, 772/4.524, 773/5.392, 774/4.858, 775/5.833, 776/5.087, 777/6.375, 778/4.858, 779/7.285, 780/4.954, 781/5.833, 782/4.954, 783/6.044, 784/4.807, 785/5.977, 786/4.858, 787/7.426, 788/4.755, 789/4.907, 790/5.044, 791/7.285, 792/4.392, 793/5.392, 794/4.755, 795/6.524, 796/4.170, 797/5.644, 798/5.248, 799/6.728, 800/4.392, 801/5.426, 802/4.755, 803/6.615, 804/3.807, 805/5.000, 806/4.700, 807/7.577, 808/4.392, 809/5.358, 810/4.000, 811/7.033, 812/4.170, 813/4.807, 814/4.248, 815/7.375, 816/3.322, 817/5.728, 818/4.585, 819/6.672, 820/3.700, 821/5.044, 822/4.907, 823/7.267, 824/4.000, 825/5.248, 826/3.700, 827/6.741, 828/4.087, 829/5.209, 830/4.170, 831/6.459, 832/4.000, 833/5.087, 834/4.755, 835/6.267, 836/3.907, 837/4.700, 838/4.087, 839/7.637, 840/4.392, 841/4.858, 842/4.170, 843/6.087, 844/4.170, 845/5.858, 846/3.907, 847/6.066, 848/3.807, 849/5.555, 850/3.000, 851/6.755, 852/4.524, 853/5.087, 854/3.700, 855/7.600, 856/4.322, 857/4.644, 858/4.087, 859/6.285, 860/4.392, 861/5.248, 862/4.459, 863/5.615, 864/4.322, 865/4.858, 866/4.322, 867/5.248, 868/3.700, 869/5.833, 870/3.700, 871/7.629, 872/4.524, 873/4.954, 874/2.807, 875/5.907, 876/4.087, 877/4.954, 878/4.392, 879/7.781, 880/4.322, 881/5.087, 882/3.459, 883/5.672, 884/4.459, 885/5.087, 886/4.585, 887/7.066, 888/4.000, 889/5.209, 890/3.459, 891/5.585, 892/3.700, 893/5.492, 894/2.807, 895/4.392, 896/4.322, 897/5.248, 898/2.322, 899/6.615, 900/4.087, 901/6.109, 902/4.000, 903/7.435, 904/4.000, 905/4.170, 906/3.807, 907/3.700, 908/2.807, 909/4.755, 910/3.000, 911/8.443, 912/3.700, 913/4.248, 914/2.807, 915/3.700, 916/3.585, 917/4.524, 918/3.322, 919/5.459, 920/4.087, 921/4.954, 922/2.322, 923/6.170, 924/3.459, 925/4.807, 926/3.000, 927/6.170, 928/3.807, 929/5.129, 930/2.585, 931/6.087, 932/2.807, 933/4.807, 934/2.585, 935/6.066, 936/3.322, 937/4.459, 938/2.322, 939/6.229, 940/2.585, 941/4.000, 942/4.000, 943/6.895, 944/3.170, 945/3.322, 946/3.700, 947/4.807, 948/3.459, 949/5.358, 950/1.585, 951/6.555, 952/2.000, 953/4.392, 954/2.000, 955/4.322, 956/3.170, 957/5.044, 958/2.322, 959/5.209, 960/3.170, 961/5.044, 962/2.585, 963/6.267, 964/2.322, 965/4.170, 966/1.585, 967/5.129, 968/2.807, 969/4.087, 970/2.000, 971/4.907, 972/2.000, 973/2.585, 974/2.807, 975/4.644, 976/1.585, 977/3.170, 978/3.459, 979/3.585, 980/3.585, 981/4.087, 982/2.000, 983/3.585, 984/4.000, 985/3.907, 986/2.000, 987/4.392, 988/2.322, 989/4.524, 990/1.585, 991/7.755, 992/2.322, 993/4.248, 994/2.322, 995/3.170, 996/2.585, 997/4.087, 998/1.585, 999/3.585, 1001/4.392, 1002/2.585, 1003/4.322, 1004/2.000, 1005/2.585, 1006/1.585, 1007/2.322, 1008/1.585, 1009/2.807, 1010/2.000, 1011/4.248, 1012/1.000, 1013/3.000, 1014/2.585, 1015/4.459, 1016/1.000, 1017/3.322, 1018/2.322, 1019/4.644, 1020/2.000, 1021/4.170, 1022/1.000, 1023/12.85, 1024/1.000, 1025/3.000, 1026/1.000, 1027/5.087, 1028/2.000, 1029/4.000, 1030/1.585, 1031/6.229, 1032/2.322, 1033/3.459, 1034/1.000, 1035/2.322, 1036/1.000, 1037/2.807, 1038/1.585, 1039/4.858, 1041/2.322, 1043/2.585, 1044/2.322, 1045/2.322, 1046/1.585, 1047/4.000, 1048/2.000, 1049/2.322, 1050/1.000, 1051/2.585, 1052/1.000, 1053/2.322, 1054/2.000, 1055/4.000, 1056/1.000, 1057/2.585, 1058/1.000, 1059/4.000, 1060/1.585, 1061/2.807, 1062/2.000, 1063/3.807, 1064/1.000, 1065/1.585, 1066/1.000, 1067/2.322, 1068/2.000, 1069/1.585, 1070/1.585, 1071/2.000, 1073/1.585, 1074/1.000, 1075/4.000, 1077/2.585, 1078/1.585, 1079/2.807, 1080/1.000, 1081/1.000, 1082/2.585, 1083/2.807, 1084/1.000, 1086/2.322, 1087/2.322, 1088/1.000, 1089/4.087, 1091/3.000, 1092/1.585, 1093/1.585, 1094/1.585, 1095/2.585, 1096/2.807, 1097/1.000, 1098/2.000, 1099/1.000, 1100/1.000, 1101/3.459, 1102/1.000, 1103/3.459, 1104/1.000, 1105/2.322, 1106/1.000, 1107/2.807, 1108/1.000, 1109/1.000, 1110/2.000, 1111/2.585, 1113/2.000, 1114/1.000, 1115/2.000, 1116/1.000, 1117/3.322, 1119/2.322, 1121/1.585, 1122/2.000, 1123/1.000, 1125/1.585, 1127/2.000, 1129/1.585, 1131/2.585, 1132/2.322, 1133/2.000, 1134/3.000, 1135/3.170, 1136/3.700, 1137/3.907, 1138/3.700, 1139/3.700, 1140/4.392, 1141/4.459, 1142/4.644, 1143/4.087, 1144/4.755, 1145/4.248, 1146/4.524, 1147/4.755, 1148/4.644, 1149/4.087, 1150/4.585, 1151/4.954, 1152/4.459, 1153/4.248, 1154/4.585, 1155/3.585, 1156/4.087, 1157/4.000, 1158/2.585, 1159/3.459, 1160/2.807, 1161/2.585, 1162/2.807, 1163/2.807, 1164/2.000, 1165/3.000, 1166/2.585, 1167/2.585, 1168/3.000, 1169/4.322, 1170/1.585, 1171/3.170, 1172/1.000, 1173/4.907, 1174/1.000, 1175/5.426, 1176/1.000, 1177/4.170, 1179/3.907, 1181/4.858, 1183/4.170, 1185/3.807, 1186/1.000, 1187/4.459, 1189/4.170, 1190/1.585, 1191/3.459, 1192/1.000, 1193/4.248, 1195/2.322, 1196/1.000, 1197/2.322, 1199/3.322, 1201/1.585, 1202/1.000, 1203/2.807, 1204/1.000, 1205/3.000, 1207/3.807, 1208/1.585, 1209/1.000, 1211/4.524, 1213/3.000, 1214/1.000, 1215/1.585, 1216/1.585, 1217/1.585, 1219/2.322, 1220/2.585, 1222/1.000, 1223/4.087, 1225/2.000, 1226/2.000, 1227/1.585, 1228/1.000, 1229/1.585, 1231/4.392, 1232/1.000, 1235/2.000, 1238/1.000, 1239/1.000, 1240/2.322, 1241/1.000, 1243/3.322, 1244/1.000, 1245/1.000, 1249/1.000, 1255/1.585, 1258/1.000, 1259/1.000, 1262/1.000, 1264/1.000, 1267/1.585, 1273/1.000, 1275/1.000, 1279/3.170, 1288/1.585, 1293/1.000, 1294/1.585, 1301/1.000, 1302/1.000, 1312/1.000, 1315/1.000, 1356/1.000, 1362/1.585, 1371/2.000, 1380/1.000, 1389/1.000, 1392/1.585, 1395/2.000, 1404/1.000, 1407/1.000, 1413/1.000, 1419/1.000, 1425/1.000, 1431/1.000, 1437/1.000, 1439/1.585, 1443/1.000, 1461/1.000}
            {
                \draw[chart,draw=colY!70!white,fill=colY!50!white] (\xstep*\nacs,0) rectangle (\xstep*\nacs+\xstep,\graphs);
            }
            
            \foreach \nacs/\graphs [count=\i] in {0/11.08, 1/14.52, 2/14.08, 3/16.42, 4/15.38, 5/15.59, 6/15.95, 7/17.07, 8/16.00, 9/16.14, 10/15.49, 11/15.89, 12/15.93, 13/16.41, 14/15.35, 15/16.98, 16/15.24, 17/15.85, 18/15.28, 19/15.50, 20/14.95, 21/15.40, 22/14.45, 23/15.44, 24/14.76, 25/15.65, 26/14.36, 27/15.74, 28/13.92, 29/14.55, 30/14.10, 31/15.95, 32/13.98, 33/14.12, 34/13.94, 35/14.46, 36/13.28, 37/14.16, 38/13.16, 39/13.95, 40/13.24, 41/13.70, 42/12.59, 43/13.86, 44/12.75, 45/13.51, 46/12.83, 47/13.84, 48/12.28, 49/13.42, 50/12.57, 51/14.09, 52/11.86, 53/12.71, 54/12.26, 55/13.82, 56/11.98, 57/12.14, 58/11.72, 59/12.47, 60/11.56, 61/12.41, 62/11.29, 63/14.19, 64/11.03, 65/12.24, 66/11.30, 67/11.89, 68/10.96, 69/12.09, 70/10.98, 71/11.78, 72/10.76, 73/11.41, 74/10.59, 75/11.87, 76/10.40, 77/11.20, 78/10.45, 79/11.40, 80/10.13, 81/11.05, 82/10.28, 83/11.20, 84/10.19, 85/10.87, 86/10.21, 87/11.40, 88/9.642, 89/11.05, 90/10.08, 91/11.43, 92/9.320, 93/10.72, 94/9.911, 95/11.48, 96/8.986, 97/10.00, 98/9.468, 99/11.03, 100/9.127, 101/10.43, 102/9.377, 103/11.71, 104/8.886, 105/9.587, 106/8.977, 107/9.771, 108/9.326, 109/10.18, 110/8.607, 111/11.01, 112/8.516, 113/10.00, 114/8.409, 115/9.301, 116/8.562, 117/9.401, 118/8.741, 119/9.371, 120/8.285, 121/9.478, 122/8.124, 123/9.944, 124/8.285, 125/8.817, 126/8.098, 127/12.29, 128/7.845, 129/8.895, 130/8.224, 131/9.405, 132/7.340, 133/8.731, 134/7.801, 135/8.839, 136/7.827, 137/8.257, 138/7.285, 139/9.236, 140/7.516, 141/8.061, 142/7.741, 143/8.539, 144/7.285, 145/8.322, 146/7.484, 147/8.335, 148/7.418, 149/7.931, 150/7.600, 151/9.006, 152/7.219, 153/7.931, 154/7.375, 155/8.224, 156/6.615, 157/7.748, 158/6.644, 159/8.480, 160/7.229, 161/7.895, 162/7.033, 163/8.114, 164/7.109, 165/7.989, 166/6.585, 167/8.109, 168/6.687, 169/7.877, 170/6.170, 171/8.285, 172/6.209, 173/8.271, 174/6.492, 175/8.755, 176/6.129, 177/7.109, 178/6.700, 179/8.500, 180/6.539, 181/7.741, 182/5.931, 183/9.066, 184/5.954, 185/6.459, 186/6.555, 187/8.109, 188/6.508, 189/7.539, 190/5.807, 191/8.791, 192/6.267, 193/6.492, 194/5.977, 195/7.267, 196/5.700, 197/6.768, 198/6.570, 199/8.640, 200/5.524, 201/7.468, 202/5.426, 203/6.658, 204/6.476, 205/7.119, 206/5.129, 207/9.172, 208/5.585, 209/6.687, 210/4.954, 211/6.570, 212/5.209, 213/6.687, 214/5.555, 215/6.267, 216/5.755, 217/7.257, 218/5.170, 219/7.366, 220/4.954, 221/6.109, 222/5.459, 223/7.637, 224/4.585, 225/6.392, 226/4.807, 227/6.870, 228/5.087, 229/5.129, 230/4.700, 231/5.977, 232/4.087, 233/5.170, 234/5.170, 235/5.833, 236/4.459, 237/6.340, 238/4.000, 239/5.392, 240/4.322, 241/5.358, 242/3.700, 243/5.524, 244/4.459, 245/5.358, 246/4.524, 247/6.129, 248/3.322, 249/5.322, 250/3.807, 251/5.459, 252/4.170, 253/5.129, 254/4.322, 255/10.24, 256/3.907, 257/5.322, 258/4.087, 259/5.044, 260/3.585, 261/4.322, 262/3.907, 263/5.392, 264/3.459, 265/4.000, 266/3.170, 267/5.285, 268/3.000, 269/4.755, 270/4.000, 271/4.585, 272/3.459, 273/4.322, 274/3.807, 275/4.087, 276/3.170, 277/4.459, 278/4.170, 279/4.585, 280/4.392, 281/3.700, 282/4.087, 283/4.459, 284/4.248, 285/4.807, 286/4.087, 287/3.585, 288/3.322, 289/3.459, 290/3.907, 291/5.492, 292/3.585, 293/4.459, 294/3.907, 295/3.807, 296/3.170, 297/3.700, 298/2.807, 299/3.807, 300/2.807, 301/5.285, 302/4.459, 303/4.322, 304/3.807, 305/4.524, 306/4.392, 307/4.087, 308/3.000, 309/5.209, 310/3.585, 311/4.087, 312/2.807, 313/3.807, 314/2.807, 315/5.000, 316/3.170, 317/4.087, 318/3.000, 319/4.392, 320/3.585, 321/4.700, 322/3.000, 323/3.907, 324/3.170, 325/4.585, 326/3.000, 327/4.459, 328/2.585, 329/4.585, 330/2.000, 331/5.358, 332/3.585, 333/3.807, 334/3.000, 335/4.392, 336/1.000, 337/3.459, 338/2.807, 339/3.585, 340/3.170, 341/3.700, 342/3.000, 343/4.755, 344/2.322, 345/3.170, 346/3.585, 347/5.755, 348/3.170, 349/3.322, 350/3.322, 351/4.248, 352/2.585, 353/4.087, 354/2.585, 355/3.000, 356/3.322, 357/2.807, 358/4.248, 359/4.000, 360/3.585, 361/3.807, 362/2.807, 363/5.170, 364/3.700, 365/3.700, 366/2.000, 367/3.700, 368/2.322, 369/3.459, 370/2.585, 371/3.170, 372/2.807, 373/4.700, 374/2.585, 375/5.392, 376/2.807, 377/3.907, 378/2.585, 379/3.170, 380/2.322, 381/4.248, 382/1.585, 383/2.807, 384/3.170, 385/4.459, 386/1.000, 387/2.000, 388/3.000, 389/4.087, 390/2.322, 391/2.000, 392/1.585, 393/3.907, 395/4.392, 396/1.000, 397/3.459, 398/2.000, 399/3.459, 400/1.585, 401/3.170, 402/1.585, 403/4.907, 404/1.585, 405/4.170, 406/2.000, 407/4.087, 408/2.322, 409/3.322, 410/1.000, 411/4.170, 412/2.000, 413/3.585, 414/2.000, 415/2.807, 416/2.322, 417/3.459, 418/1.585, 419/4.700, 420/1.000, 421/2.585, 422/3.000, 423/2.322, 424/2.322, 425/3.170, 426/2.322, 427/4.644, 429/2.807, 430/2.000, 431/1.585, 432/1.585, 433/1.585, 434/2.322, 435/4.322, 436/1.585, 437/2.000, 438/1.585, 439/5.209, 440/1.585, 441/1.000, 442/1.000, 443/4.000, 444/1.585, 445/2.585, 446/1.585, 447/1.000, 448/1.585, 449/3.459, 450/3.000, 451/4.459, 452/1.000, 453/1.000, 455/5.555, 456/1.000, 459/2.000, 460/1.585, 461/2.585, 463/3.000, 464/2.000, 465/2.322, 466/1.000, 467/2.585, 468/1.000, 469/3.000, 471/3.585, 473/1.000, 474/2.585, 475/3.322, 476/1.000, 477/1.000, 478/1.585, 479/1.585, 480/1.000, 481/3.322, 482/1.585, 483/2.322, 484/1.000, 485/1.585, 487/1.000, 489/1.000, 490/1.000, 493/1.000, 494/1.585, 495/4.954, 496/1.000, 498/1.000, 500/1.000, 501/1.000, 505/1.000, 507/1.000, 509/1.585, 510/1.000, 511/9.687, 513/2.000, 514/1.585, 515/3.170, 519/2.000, 523/1.000, 527/1.000, 529/1.000, 531/1.585, 537/1.000, 541/1.000, 544/1.585, 550/1.000, 551/1.000, 553/1.000, 558/1.000, 579/1.000, 580/1.000, 583/1.000, 584/2.322, 585/1.000, 586/2.585, 587/3.000, 588/2.000, 589/1.585, 590/2.322, 591/1.585, 592/1.000, 593/1.585, 594/1.000, 601/1.000, 602/1.000, 603/1.585, 605/2.322, 606/1.000, 611/1.585, 615/2.322, 621/1.000, 639/1.000}
            {
                \draw[chart,draw=colG!70!white,fill=colG!50!white] (\xstep*\nacs,0) rectangle (\xstep*\nacs+\xstep,\graphs);
            }
        }
    \end{tikzpicture}
    
    \caption{The numbers of minimally rigid graphs on 11 (green) and 12 vertices (yellow)
    according to the number of NAC-colourings.
    The numbers were determined computationally using the code supporting~\cite{LL2024} and the generator of minimally rigid graphs \cite{Larsson}.
    }
    \label{fig:NACnumbers_all_11}
\end{figure}

\begin{figure}[ht]
    \centering%
    \begin{tikzpicture}[yscale=2.5,xscale=1.2]
        \draw[white] (-1.5,0) circle (2pt);
    \draw[white] (1.5,0) circle (2pt);
        \node[vertex] (0) at (1, 0.2) {};
        \node[vertex] (1) at (-0.5, 0.5) {};
        \node[vertex] (2) at (-0.5, -0.5) {};
        \node[vertex] (3) at (-1, 0.2) {};
        \node[vertex] (4) at (0.5, -0.5) {};
        \node[vertex] (5) at (0.5, 0.5) {};
        \node[vertex] (6) at (-0.5, 0) {};
        \node[vertex] (7) at (0.5, 0) {};
        \draw[edge](0)edge(4) (0)edge(5) (0)edge(7) (1)edge(3) (1)edge(5) (1)edge(7) (2)edge(3) (2)edge(4) (2)edge(7) (3)edge(6) (4)edge(6) (5)edge(6) (6)edge(7);
    \end{tikzpicture}\qquad
    \begin{tikzpicture}[rotate=90,xscale=2.5/3.7,yscale=1.2]
        \draw[white] (0,-1.5) circle (2pt);
    \draw[white] (0,1.5) circle (2pt);
        \node[vertex] (0) at (-0.7, 0) {};
        \node[vertex] (1) at (0, -1) {};
        \node[vertex] (2) at (0, 1) {};
        \node[vertex] (3) at (0.4, 0) {};
        \node[vertex] (4) at (3, 0.5) {};
        \node[vertex] (5) at (3, -0.5) {};
        \node[vertex] (6) at (2, 1) {};
        \node[vertex] (7) at (2, -1) {};
        \node[vertex] (8) at (1.6, 0) {};
        \draw[edge](0)edge(1) (0)edge(2) (0)edge(3) (1)edge(7) (1)edge(8) (2)edge(6) (2)edge(8) (3)edge(6) (3)edge(7) (4)edge(6) (4)edge(7) (4)edge(8) (5)edge(6) (5)edge(7) (5)edge(8);
    \end{tikzpicture}\qquad
    \begin{tikzpicture}[yscale=2.5/1.6,xscale=1.2]
        \draw[white] (-1.5,1) circle (2pt);
    \draw[white] (1.5,1) circle (2pt);
        \node[vertex] (0) at (1, 1.5) {};
        \node[vertex] (1) at (-1, 1.5) {};
        \node[vertex] (2) at (0, 1.8) {};
        \node[vertex] (3) at (-1, 0.4) {};
        \node[vertex] (4) at (-0.6, 1.4) {};
        \node[vertex] (5) at (1, 0.4) {};
        \node[vertex] (6) at (0.6, 1.4) {};
        \node[vertex] (7) at (0, 0.2) {};
        \node[vertex] (8) at (0, 1) {};
        \draw[edge](0)edge(2) (0)edge(5) (0)edge(6) (1)edge(2) (1)edge(3) (1)edge(4) (2)edge(8) (3)edge(7) (3)edge(8) (4)edge(7) (4)edge(8) (5)edge(7) (5)edge(8) (6)edge(7) (6)edge(8);
    \end{tikzpicture}

    \vspace{1cm}
    \begin{tikzpicture}[xscale=9/14]
            \node[vertex] (0) at (1.0, 0) {};
            \node[vertex] (1) at (1.0, 1.0) {};
            \node[vertex] (2) at (2.0, 0) {};
            \node[vertex] (3) at (2.0, 1.0) {};
            \node[vertex] (4) at (3.0, 0) {};
            \node[vertex] (5) at (3.0, 1.0) {};
            \node[vertex] (8) at (4.0, 0) {};
            \node[vertex] (9) at (4.0, 1.0) {};
            \node[vertex] (12) at (-0.3, -1.0) {};
            \node[vertex] (13) at (5.3, -1.0) {};
            \draw[edge] (0)edge(2) (0)edge(3) (0)edge(12) (1)edge(2) (1)edge(3);
            \draw[edge] (1)edge(12) (2)edge(4) (2)edge(5) (3)edge(4) (3)edge(5);
            \draw[edge] (8)edge(4) (8)edge(5) (9)edge(4) (9)edge(5) (8)edge(13) (9)edge(13) (12)edge(13);
        \end{tikzpicture}\qquad
            \begin{tikzpicture}[scale=0.3]
	\node[vertex] (0) at (0, 7) {};
	\node[vertex] (1) at (0, 2) {};
	\node[vertex] (2) at (-3, 2) {};
	\node[vertex] (3) at (4, 0) {};
	\node[vertex] (4) at (4, 4) {};
	\node[vertex] (5) at (6, 5) {};
	\node[vertex] (6) at (-6, 5) {};
	\node[vertex] (7) at (-4, 0) {};
	\node[vertex] (8) at (-4, 4) {};
	\node[vertex] (9) at (1, 4) {};
	\node[vertex] (10) at (1, 0) {};
	\draw[edge] (0) to (5) (0) to (6) (1) to (2) (1) to (9) (1) to (10) (2) to (7) (2) to (8) (3) to (5) (3) to (9) (3) to (10) (4) to (5) (4) to (9) (4) to (10) (6) to (7) (6) to (8) (7) to (9) (7) to (10) (8) to (9) (8) to (10);
    \end{tikzpicture}\qquad
    \begin{tikzpicture}[scale=0.3]
    \draw[white] (-6,0) circle (2pt);
    \draw[white] (6,0) circle (2pt);
    	\node[vertex] (0) at (0, 1) {};
    	\node[vertex] (1) at (-2, -2) {};
    	\node[vertex] (2) at (-3, -6) {};
    	\node[vertex] (3) at (-3, -3) {};
    	\node[vertex] (4) at (2, -2) {};
    	\node[vertex] (5) at (3, -3) {};
    	\node[vertex] (6) at (3, -6) {};
    	\node[vertex] (7) at (0, -1) {};
    	\node[vertex] (8) at (-5, -2) {};
    	\node[vertex] (9) at (5, -2) {};
    	\node[vertex] (10) at (0, -3) {};
    	\node[vertex] (11) at (0, -6) {};
    	\draw[edge] (0) to (7) (0) to (8) (0) to (9) (1) to (7) (1) to (10) (1) to (11) (2) to (8) (2) to (10) (2) to (11) (3) to (8) (3) to (10) (3) to (11) (4) to (9) (4) to (10) (4) to (11) (5) to (9) (5) to (10) (5) to (11) (6) to (9) (6) to (10) (6) to (11);
    \end{tikzpicture}
    \caption{Some graphs maximising the number of NAC-colourings among minimally rigid graphs on $8$ (top left), $9$ (top middle and right), $10$ (bottom left), $11$ (bottom middle) and $12$ vertices (bottom right). On $7$ vertices the unique graph attaining the maximum is $K_{3,3}$ with one open $0$-extension.
    }
    \label{fig:maxNACs}
\end{figure}
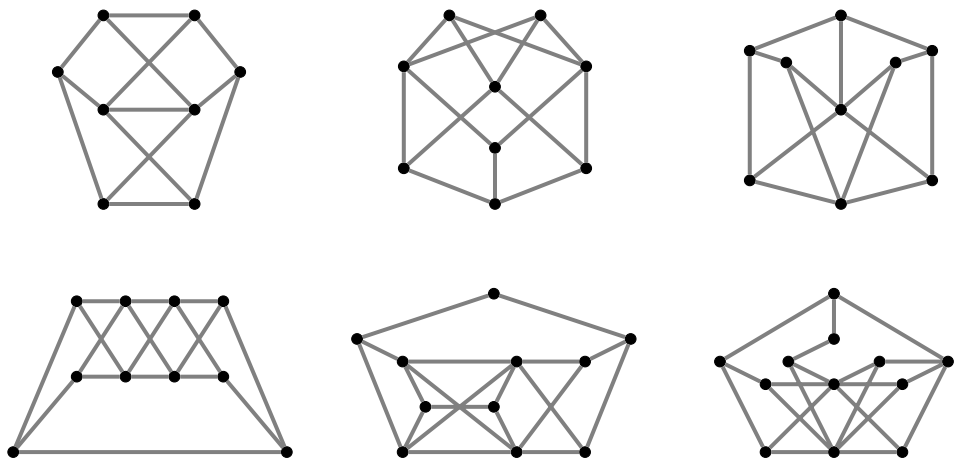

For each $k \geq 1$, we let $G_k$ be the graph with vertex set $\{x,y\} \cup \{a_1, \dots, a_k\} \cup \{b_1, \dots, b_k\}$ and edges $a_ia_{i+1},\allowbreak b_ib_{i+1},\allowbreak a_ib_{i+1},\allowbreak b_ia_{i+1}$ for all $i \in [k-1]$ and $xy,\allowbreak xa_1,\allowbreak xb_1,\allowbreak ya_1,\allowbreak yb_1$; see \Cref{fig:minRigidWithManyNACs}.
It is easy to check that $G_k$ has $2k+2$ vertices and $4k+1$ edges, and that $G_k$ is a minimally rigid $0$-extension graph.

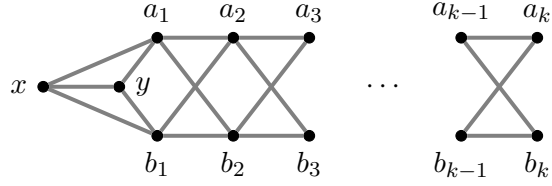
\begin{figure}[ht]
        \centering
        \begin{tikzpicture}[yscale=1.3]
            \node[vertex,label={left:$x$}] (12) at (-0.5,0.5) {};
            \node[vertex,label={right:$y$}] (13) at (0.5,0.5) {};
            \node[vertex,label={below:$b_1$}] (0) at (1.0, 0.00) {};
            \node[vertex,label={above:$a_1$}] (1) at (1.0, 1.0) {};
            \node[vertex,label={below:$b_2$}] (2) at (2.0, 0.00) {};
            \node[vertex,label={above:$a_2$}] (3) at (2.0, 1.0) {};
            \node[vertex,label={below:$b_3$}] (4) at (3.0, 0.00) {};
            \node[vertex,label={above:$a_3$}] (5) at (3.0, 1.0) {};
            \node (6) at (4, 0.5) {$\cdots$};
            \node[vertex,label={below:$b_{k-1}$}] (8) at (5.0, 0.00) {};
            \node[vertex,label={above:$a_{k-1}$}] (9) at (5.0, 1.0) {};
            \node[vertex,label={below:$b_k$}] (10) at (6.0, 0.00) {};
            \node[vertex,label={above:$a_k$}] (11) at (6.0, 1.0) {};
            \draw[edge] (0)edge(2) (0)edge(3) (1)edge(2) (1)edge(3);
            \draw[edge] (2)edge(4) (2)edge(5) (3)edge(4) (3)edge(5);
            \draw[edge] (8)edge(10) (8)edge(11) (9)edge(10) (9)edge(11);
            \draw[edge] (0)edge(12) (0)edge(13) (1)edge(12) (1)edge(13) (12)edge(13);
        \end{tikzpicture}
        \caption{The graph $G_k$.}
        \label{fig:minRigidWithManyNACs}
\end{figure}

\begin{proposition}\label{prop:minrigidNACbound}
    Let $k \geq 2$ be an integer, and let $n = 2k+2 = |V(G_k)|$, where $G_k$ is defined above. We have $\nnac{G_k} = 2^{n-4}-1$.   
\end{proposition}

\begin{proof}
It is easy to verify the formula for $k=2$, so we henceforth assume $k\geq 3$.
    Let $G_k':=G_k \setminus \{x,y\}$.
    We say that a colouring $c\colon E(G_k') \to \{\red, \blue\}$ is \emph{locally NAC}
    if $c$ restricted to $G_k[\{a_i, a_{i+1}, a_{i+2}, b_i, b_{i+1}, b_{i+2} \}]$ has no almost-monochromatic cycle for all $i \in [k-2]$. Note that this requires an even number of $a_ia_{i+1}$, $a_ib_{i+1}$, $b_ia_{i+1}$ and $b_ib_{i+1}$ to have each colour for each $i<k$, and that for each such colouring of these four edges, there are $4$ ways to colour the next four edges so that $G_k[\{a_i, a_{i+1}, a_{i+2}, b_i, b_{i+1}, b_{i+2} \}]$ has no almost-monochromatic cycle
    (see \Cref{fig:localNAC}).

\begin{figure}[ht]
    \centering
        \begin{tikzpicture}
        \begin{scope}
            \node[vertex] (a1) at (0, 1) {};
            \node[vertex] (a2) at (1, 1) {};
            \node[vertex] (b1) at (0, 0) {};
            \node[vertex] (b2) at (1, 0) {};
            \draw[bedge] (a1)edge(a2) (a1)edge(b2) (b1)edge(b2) (b1)edge(a2);
            \node (A) at (1.5,0.5) {or};
            \coordinate (A2) at (1.5,-0.5);
            \draw [edge,->] (-0.5,0.45) to[out=-120, in=-90,looseness=1.5] (-1.5,0.5) to[out=90, in=120,looseness=1.5] (-0.5,0.55);
            \begin{scope}[xshift=2cm]
                \node[vertex] (a1) at (0, 1) {};
                \node[vertex] (a2) at (1, 1) {};
                \node[vertex] (b1) at (0, 0) {};
                \node[vertex] (b2) at (1, 0) {};
                \draw[redge] (a1)edge(a2) (a1)edge(b2) (b1)edge(b2) (b1)edge(a2);
                \coordinate (A1) at (1.5,0.5);
            \end{scope}
        \end{scope}
        \begin{scope}[xshift=5cm]
            \node[vertex] (a1) at (0, 1) {};
            \node[vertex] (a2) at (1, 1) {};
            \node[vertex] (b1) at (0, 0) {};
            \node[vertex] (b2) at (1, 0) {};
            \draw[bedge] (a1)edge(b2) (b1)edge(b2);
            \draw[redge] (a1)edge(a2) (b1)edge(a2);
            \node (B) at (1.5,0.5) {or};
            \coordinate (B1) at (-0.5,0.5);
            \coordinate (B2) at (1.5,-0.5);
            \coordinate (B3) at (-0.5,-0.5);
            \begin{scope}[xshift=2cm]
                \node[vertex] (a1) at (0, 1) {};
                \node[vertex] (a2) at (1, 1) {};
                \node[vertex] (b1) at (0, 0) {};
                \node[vertex] (b2) at (1, 0) {};
                \draw[redge] (a1)edge(b2) (b1)edge(b2);
                \draw[bedge] (a1)edge(a2) (b1)edge(a2);
            \end{scope}
        \end{scope}
        \begin{scope}[yshift=-3cm]
            \node[vertex] (a1) at (0, 1) {};
            \node[vertex] (a2) at (1, 1) {};
            \node[vertex] (b1) at (0, 0) {};
            \node[vertex] (b2) at (1, 0) {};
            \draw[bedge] (a1)edge(a2) (a1)edge(b2);
            \draw[redge] (b1)edge(b2) (b1)edge(a2);
            \coordinate (D2) at (1.5,1.5);
            \node (D) at (1.5,0.5) {or};
            \begin{scope}[xshift=2cm]
                \node[vertex] (a1) at (0, 1) {};
                \node[vertex] (a2) at (1, 1) {};
                \node[vertex] (b1) at (0, 0) {};
                \node[vertex] (b2) at (1, 0) {};
                \coordinate (D1) at (1.5,0.5);
                \coordinate (D3) at (1.5,1.5);
            \draw[redge] (a1)edge(a2) (a1)edge(b2);
            \draw[bedge] (b1)edge(b2) (b1)edge(a2);
            \end{scope}
        \end{scope}
        \begin{scope}[xshift=5cm,yshift=-3cm]
            \node[vertex] (a1) at (0, 1) {};
            \node[vertex] (a2) at (1, 1) {};
            \node[vertex] (b1) at (0, 0) {};
            \node[vertex] (b2) at (1, 0) {};
            \draw[bedge] (a1)edge(a2) (b1)edge(b2);
            \draw[redge] (a1)edge(b2) (b1)edge(a2);
            \coordinate (C1) at (-0.5,0.5);
            \coordinate (C2) at (1.5,1.5);
            \node (C) at (1.5,0.5) {or};
            \begin{scope}[xshift=2cm]
                \node[vertex] (a1) at (0, 1) {};
                \node[vertex] (a2) at (1, 1) {};
                \node[vertex] (b1) at (0, 0) {};
                \node[vertex] (b2) at (1, 0) {};
                \draw[redge] (a1)edge(a2) (b1)edge(b2);
                \draw[bedge] (a1)edge(b2) (b1)edge(a2);
                \draw [edge,->] (1.5,0.55) to[out=60, in=90,looseness=1.5] (2.5,0.5) to[out=-90, in=-60,looseness=1.5] (1.5,0.45);
            \end{scope}
        \end{scope}
        \draw[edge,->] (A1) to (B1);
        \draw[edge,->] (D2) to (A2);
        \draw[edge,->] (B2) to (C2);
        \draw[edge,->] (C1) to (D1);
        \draw[edge,<->] (B3) to (D3);
    \end{tikzpicture}
    \caption{The diagram shows how two consecutive parts of the graph $G'_k$ can be coloured.}
    \label{fig:localNAC}
\end{figure}
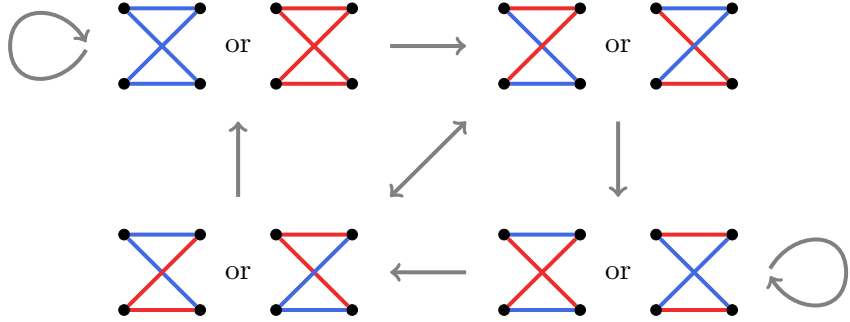
    
    \begin{claim}\label{claim:local}
        Every locally NAC-colouring of $G_k'$ is a NAC-colouring of $G_k'$ or monochromatic.
    \end{claim}
    \begin{proof}[Proof of \Cref{claim:local}]
        We proceed by induction on $k$.  For $k=3$ a colouring is locally NAC if and only if it has no almost-monochromatic cycle. Assume $k \geq 4$. Towards a contradiction, suppose $G_k'$ has a locally NAC-colouring $c'$ which is neither monochromatic nor a NAC-colouring, i.e.\ it has an almost-monochromatic cycle $C$. By induction, $G_k\setminus\{a_1,b_1\}$ and $G_k\setminus\{a_k,b_k\}$ have no such cycle, and so $C$ contains
        a vertex from $\{a_1,b_1\}$ and from $\{a_k,b_k\}$. Without loss of generality, we may assume $C$ contains a unique red edge $e$ and (by exchanging labels on each pair $a_i,b_i$ if necessary) is the cycle $a_1a_2\cdots a_kb_{k-1}b_{k-2}\cdots b_2a_1$.
        
        Suppose $e \in \{a_1a_2, a_1b_2\}$.  Then  $a_ka_{k-1}$, $a_kb_{k-1}$ and $b_{k-2}b_{k-1}$ are all blue. Since $c'$ is locally NAC, $b_{k-2}a_{k-1}$ must also be blue. But this gives an almost-blue cycle $a_1a_2\cdots a_{k-1}b_{k-2}\cdots b_2a_1$ avoiding $a_k$, which is a contradiction. Similarly if $e \in \{a_{k-1}a_k, b_{k-1}b_k\}$ we obtain a contradiction.
        
        Thus $e=a_ia_{i+1}$ or $e=b_ib_{i+1}$ for some $2\leq i\leq k-2$. Consider the two edges $a_ib_{i+1}$ and~$b_ia_{i+1}$. Both are chords of $C$, and if either is blue then one of the two cycles formed by adding that edge to $C$ is almost-blue and avoids $a_1$ or $a_k$, a contradiction. So both are red, but then $a_ia_{i+1}b_ib_{i+1}a_i$ is almost-red, and $c'$ is not a locally NAC-colouring.
    \end{proof}
    
    Consider all locally NAC-colourings $c$ of $G_k'$
    such that $c(a_1a_2)=c(b_1a_2)$ (and hence also $c(a_1b_2)=c(b_1b_2)$).
    Note that there are $4^{k-1}$ of them and each of them except for the whole blue colouring
    extends to a NAC-colouring of $G_k$ by colouring the edges incident to $x$ or $y$ blue.
    
    Since every NAC-colouring of $G_k$ restricts to a locally NAC-colouring $c$ of $G_k'$ with $c(a_1a_2)=c(b_1a_2)$,
    $\nnac{G_k} = 4^{k-1} - 1 =2^{2k-2}-1=2^{|V(G_k)|-4}-1$. 
\end{proof}

We next show that there are minimally rigid graphs with significantly more NAC-colourings. One way to prove this is to start from $G_k$, for sufficiently large $k$, and perform $0$-extensions on $(a_1,a_{r+1})$, $(a_{r+1},a_{2r+1})$, and so on, where $r$ is some fixed integer. It is not hard to show, given the classification of NAC-colourings of $G_k$, that provided $r$ is not too small, these $0$-extensions each increase the number of NAC-colourings by a factor greater than $2$. Since the number of $0$-extensions is a constant proportion of the number of vertices, this gives an exponential lower bound where the base of the exponent is strictly greater than $2$. However, the largest value we were able to obtain by this method is still significantly lower than the one we obtain below using an entirely different method.

\begin{lemma}\label{lem:gluing}
    Let $G$ be a graph obtained by gluing $k$ copies of a graph $H$ along an edge.
    Then
    \[
        \nnac{G} = (\nnac{H} + 1)^k - 1 = (\nnac{H} + 1)^\frac{|V(G)|-2}{|V(H)|-2} - 1\,.
    \]
\end{lemma}
\begin{proof}
    We fix the colour of the common edge $e$ to be blue.
    Every NAC-colouring of $G$ with $e$ being blue restricted to any copy of $H$ is a NAC-colouring or monochromatic.
    On the other hand, every surjective colouring of $G$ such that each copy of $H$ is monochromatic or a NAC-colouring is a NAC-colouring of $G$,
    since every cycle is within a copy of $H$ or has $e$ as a chord.
    Therefore, the first equality holds. The second follows from $|V(G)|=k(|V(H)|-2)+2$.
\end{proof}
Since gluing two minimally rigid graphs along an edge gives a minimally rigid graph, by taking $H$ to be the minimally rigid graph in \Cref{fig:18vert}, and noting that $\sqrt[16]{180\,608}>2.13$,
we get the following.
\begin{corollary}\label{cor:2+}
    There is an infinite family of minimally rigid graphs $H_k$ such that $\nnac{H_k} = \omega\left(2.13^{|V(H_k)|}\right)$.
\end{corollary}
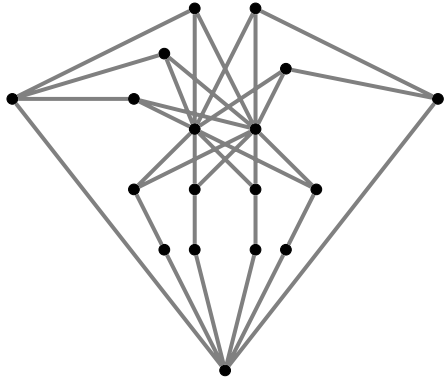
\begin{figure}[ht]
    \centering
    \begin{tikzpicture}[scale=0.4]
% Framework in 2-dimensional space consisting of:
% Graph with vertices [0, 1, 2, 3, 4, 5, 6, 7, 8, 9, 10, 11, 12, 13, 14, 15, 16, 17] and edges [[0, 1], [0, 5], [0, 9], [1, 2], [1, 4], [1, 6], [1, 10], [1, 11], [1, 12], [1, 14], [1, 16], [2, 3], [2, 5], [3, 4], [3, 8], [3, 11], [4, 5], [5, 6], [5, 10], [5, 11], [5, 12], [5, 14], [5, 16], [6, 7], [7, 8], [8, 9], [8, 13], [8, 15], [8, 17], [9, 10], [12, 13], [14, 15], [16, 17]]
% Realization {0:(1, 5), 1:(-1, 1), 2:(-2, 3.5), 3:(-7, 2), 4:(-1, 5), 5:(1, 1), 6:(1, -1), 7:(1, -3), 8:(0, -7), 9:(7, 2), 10:(2, 3), 11:(-3, 2), 12:(-1, -1), 13:(-1, -3), 14:(-3, -1), 15:(-2, -3), 16:(3, -1), 17:(2, -3)}
    	\node[vertex] (0) at (1, 5) {};
    	\node[vertex] (1) at (-1, 1) {};
    	\node[vertex] (2) at (-2, 3.5) {};
    	\node[vertex] (3) at (-7, 2) {};
    	\node[vertex] (4) at (-1, 5) {};
    	\node[vertex] (5) at (1, 1) {};
    	\node[vertex] (6) at (1, -1) {};
    	\node[vertex] (7) at (1, -3) {};
    	\node[vertex] (8) at (0, -7) {};
    	\node[vertex] (9) at (7, 2) {};
    	\node[vertex] (10) at (2, 3) {};
    	\node[vertex] (11) at (-3, 2) {};
    	\node[vertex] (12) at (-1, -1) {};
    	\node[vertex] (13) at (-1, -3) {};
    	\node[vertex] (14) at (-3, -1) {};
    	\node[vertex] (15) at (-2, -3) {};
    	\node[vertex] (16) at (3, -1) {};
    	\node[vertex] (17) at (2, -3) {};
    	\draw[edge] (0) to (1) (0) to (5) (0) to (9) (1) to (2) (1) to (4) (1) to (6) (1) to (10) (1) to (11) (1) to (12) (1) to (14) (1) to (16) (2) to (3) (2) to (5) (3) to (4) (3) to (8) (3) to (11) (4) to (5) (5) to (6) (5) to (10) (5) to (11) (5) to (12) (5) to (14) (5) to (16) (6) to (7) (7) to (8) (8) to (9) (8) to (13) (8) to (15) (8) to (17) (9) to (10) (12) to (13) (14) to (15) (16) to (17);
    \end{tikzpicture}
    \caption{An $18$-vertex minimally-rigid graph $H$ with $\nnac{H}=180\,607$ (computed using~\cite{LL2024}),
    which can be obtained by $0$-extensions from the $12$-vertex graph in \Cref{fig:maxNACs}.}
    \label{fig:18vert}
\end{figure}

\subsection{Flexible graphs}

To close this section, we consider the number of NAC-colourings in flexible graphs. Note that a flexible graph clearly has a flexible quasi-injective realisation, and hence by \Cref{thm:NACiffflex}, such a graph always has at least one NAC-colouring. We can use \Cref{prop:flexibleHasStableSeparator} to derive a stronger bound. 

\begin{lemma}\label{lem:logstableseps}
    Let $G$ be a $2$-connected flexible graph with $\ell$ rigid components.  Then $G$ has at least  $\lceil \log_2 \ell \rceil$ stable separations, and hence $\nnac{G} \geq \lceil \log_2 \ell \rceil$.
\end{lemma}

In the proof we will use the \emph{biclique covering number}; this is the minimum number of complete bipartite subgraphs needed to cover the edge set of a graph.

\begin{proof}
    Let $R_1, \dots, R_\ell$ be the rigid components of $G$. We say that a stable separation $\{G_1, G_2\}$ is \emph{respectful} if $V(G_1 \cap G_2)$ contains at most one vertex from each $R_i$.  Note that each $R_i$ is either $2$-connected or consists of a single edge, and hence if $\{G_1, G_2\}$ is respectful, then each $R_i$ is contained in either $G_1$ or $G_2$. 
    
    Let $K_\ell$ be the complete graph with vertex set $[\ell]$. We say that $ij \in E(K_\ell)$ is \emph{covered} by a respectful separation $\{G_1, G_2\}$ if $R_i \subseteq G_1$ and $R_j \subseteq G_2$. The edges covered by $\{G_1,G_2\}$ induce a complete bipartite subgraph of $K_\ell$.  By \Cref{prop:flexibleHasStableSeparator}, every $ij \in E(K_\ell)$ is covered by a respectful separation.  Therefore, the number of respectful separations of $G$ is at least the biclique covering number of $K_\ell$. It is well-known that this number is $\lceil \log_2 \ell \rceil$ (see~\cite{FH96} for a proof). 
    
    Finally, the bound on $\nnac{G}$ follows from the one-to-one correspondence between stable separations and NAP-colourings given by \Cref{lem:NAPiffStableCut}. 
\end{proof}

For flexible graphs with few rigid components, we can deduce the following improved bound on the number of NAC-colourings. In the proof, we will use the well-known fact that if $R_1,R_2,R_3$ are distinct, pairwise intersecting rigid components of a graph $G$, then they must intersect at a single common vertex. It follows that a $2$-connected flexible graph has at least four rigid components. 

\begin{lemma}\label{lem:flexibleNAC3}
    Let $G$ be a $2$-connected flexible graph. Then $\nnac{G}\geq 3$.
\end{lemma}

\begin{proof}
    Let $R_1,\ldots,R_\ell$ denote the rigid components of $G$; by the observation in the previous paragraph, we have $\ell \geq 4$. If $\ell \geq 5$, then the bound given by \Cref{lem:logstableseps} suffices.
    
    Hence, we may assume that $\ell = 4$. By using the observation above once more, as well as possibly reordering $R_1,\ldots,R_4$, we may also assume that $R_i \cap R_{i+1} \neq \varnothing$ and $R_i \cap R_{i+2} = \varnothing$ for all $i \in [4]$, where subscripts are taken cyclically. It is not difficult to deduce now that any colouring for which two of $E(R_1), \dots, E(R_4)$ are completely red while the other two are completely blue is a NAC-colouring. Therefore, $\nnac{G} \geq 3$.  
\end{proof}

The $2$-connectivity assumption cannot be removed in \Cref{lem:flexibleNAC3}: gluing two complete graphs along a single vertex yields a flexible graph with a unique NAC-colouring, up to swapping colours. More generally, the following simple observation expresses the number of NAC-colourings of a graph in terms of the number of NAC-colourings of its blocks.
Recall that a \emph{block} in a graph $G$ is either a maximal $2$-connected subgraph of $G$ or a bridge.  Note that under this definition, an isolated vertex is not a block. Moreover, adding isolated vertices to a graph clearly does not change its number of NAC-colourings.  

\begin{lemma} \label{lem:blockproduct}
    Let $G$ be a graph and $B_1, \dots, B_k$ be the blocks of $G$.  Then 
    \[
    \nnac{G}=\frac{1}{2}\prod_{i=1}^k (2\nnac{B_i}+2)-1
    \]
\end{lemma}

\begin{proof}
    Let $C$ be the set of colourings $c\colon E(G) \to \{\red, \blue\}$ such that the restriction of $c$ to each $B_i$ is either monochromatic or a NAC-colouring.  Note that $|C|=\prod_{i=1}^k (2\nnac{B_i}+2)$.  Every colouring in $c$ is a NAC-colouring of $G$ except the two monochromatic colourings.  Thus, 
\[
    \nnac{G} \geq \frac{|C|-2}{2}=\frac{1}{2}\prod_{i=1}^k (2\nnac{B_i}+2)-1.
\]
    For the other inequality, observe that every NAC-colouring of $G$ restricts to either a NAC-colouring or a monochromatic colouring on each $B_i$.  
\end{proof}

By combining the previous results, we can deduce the following description of the family of flexible graphs for which $\nnac{G}$ is minimal.

\begin{proposition}\label{prop:flexibleOneNacChar}
    Let $G$ be a flexible graph. Then $\nnac{G}=1$  if and only if $G$ has a  separation $\{G_1,G_2\}$ with $G_1\cap G_2=K_1$ and $\nnac{G_i}=0$ for each $i \in \{1,2\}$.
\end{proposition}
\begin{proof}
    Suppose $\nnac{G}=1$.  Let $B_1,B_2,\dots, B_k$ be the blocks of $G$.
    By \Cref{lem:flexibleNAC3}, $k \neq 1$.  By~\Cref{lem:blockproduct}, $k=2$ and neither $B_1$ nor $B_2$ has a NAC-colouring. 
    
    For the other implication, suppose $G$ has a separation $\{G_1,G_2\}$ with $G_1\cap G_2=K_1$  and $\nnac{G_i}=0$ for each $i \in \{1,2\}$.  Since $G_i$ has no NAC-colouring, it has no stable cuts, and in particular, it is $2$-connected.  Therefore, each $G_i$ must be a block of $G$. Hence $\nnac{G}=1$ by~\Cref{lem:blockproduct}.
\end{proof}

%================================================
\section{Concluding remarks}
\label{sec:conclusion}
%================================================

We finish the paper by highlighting some open problems. The first one concerns graphs $G$ with $\nnac{G} = 1$. 
From \Cref{prop:flexibleOneNacChar}, it is not difficult to deduce that the problem of deciding whether $\nnac{G} = 1$ holds for a flexible graph $G$ is equivalent to the problem of deciding whether $\nnac{H} = 0$ holds for a general graph $H$; in particular, it is coNP-hard. We can deduce the same hardness result for rigid graphs by combining the hardness proof for the existence of a NAC-colouring in \cite{Garamvolgyi2022} with the fact that the unique satisfiability problem is coNP-hard \cite{blass.gurevich_1982}. 

Thus, let us focus on the case of minimally rigid graphs.
In this setting, we pose the following conjecture, which we have verified for graphs on up to $10$ vertices by computer.

\begin{conjecture}
    \label{conj:uniqueNACmindegTwo}
    Let $G$ be a minimally rigid graph.
    We have $\nnac{G}=1$ if and only if
    \begin{enumerate}
        \item $G\in\noStableCutThree$ and it has exactly one $3$-prism subgraph, or
        \item $G$ is a $0$-extension graph with exactly one open step.
    \end{enumerate}
\end{conjecture}

It is not difficult to verify that the graphs appearing  in \Cref{conj:uniqueNACmindegTwo} satisfy $\nnac{G} = 1$. More generally, for each graph $G \in \noStableCutThree$, we have $\nnac{G} = 2^p - 1$, where $p$ is the number of $3$-prism subgraphs used in the recursive construction of $G$.

In the opposite direction, it is still open to describe more precisely the growth rate of
\[
    M_n=\max\{\nnac{G}:G\text{ minimally rigid}, |V(G)|=n\}\,.
\]
In particular, what is $\lim_{n\to\infty} M_n^{1/n}$? It follows from \Cref{lem:gluing} that the limit exists, and from \Cref{thm:stableCutImpliesNAC} and \Cref{cor:2+} that it is greater than $2.13$ but at most $4$.
It seems likely that with more computation a better base graph to improve the lower bound will be found, and so the key question is whether the upper bound is tight.

Another interesting open problem is to determine the complexity of the following problems: given a minimally rigid graph $G$ and a pair of edges $e,f \in E(G)$, is there a NAC-colouring of $G$ in which $e$ and $f$ receive different colours (resp.\ the same colour)? Using the connection between NAC-colourings and motions, this is equivalent to the problem of deciding whether $G$ has a motion in which the angle of $e$ and $f$ changes (resp.\ remains constant).

Let us also highlight that for graphs with $n$ vertices and $2n-2$ edges, both deciding the existence of a NAC-colouring and deciding the existence of a stable cut is open. Both of these problems appear to be challenging.

Recall that \Cref{thm:chenyu} tells us that every graph on $n$ vertices and at most $2n-4$ edges has a stable cut. Chernyshev, Rauch and  Rautenbach \cite{chernyshev.etal_2024} recently conjectured that every graph $G$ on $n$ vertices and at most $3n-7$ edges has a cut $X$ such that $G[X]$ is a forest. Since we showed in \Cref{cor:flexiblestablecut} that every $2$-flexible graph has a stable cut, it is tempting to conjecture that every $3$-flexible graph has a forest cut. However, this is not the case: the Cartesian product of $K_2$ and $K_5$ is $3$-flexible, but has no forest cut.

Finally, the natural problem of determining the threshold probability for a binomial random graph to have no NAC-colouring was recently settled by a subset of the present authors \cite{thresholds}, who also showed that this threshold coincides with the threshold for having no stable cut. However, this leaves open other directions in random graphs such as estimating the number of NAC-colourings below this threshold, or determining the threshold for a forest cut.

\section*{Acknowledgements}

This project originated from the Fields Institute Focus Program on Geometric Constraint Systems in Toronto and a substantial part of the work was done during a Research-in-Groups Programme funded by the International Centre for Mathematical Sciences, Edinburgh.
The authors are grateful to both institutions for their hospitality and generous financial support. 

The authors would like to thank P.\ Laštovička for help with the code for calculating the number of NAC-colourings of minimally rigid graphs
and implementing \Cref{alg:stableCutFlexible} in \textsc{PyRigi},
and S.\ Villányi for the counterexample to the possible conjecture in the second to last paragraph.

K.\,C.\ was supported by the Australian Government through the Australian Research
Council’s Discovery Projects funding scheme (project DP210103849). 
D.\,G.\ was supported by the Lend\"ulet Programme of the Hungarian Academy of Sciences, grant number LP2021-1/2021.
T.\,H.\ was supported by the Institute for Basic Science (IBS-R029-C1).
J.\,L.\ was supported by the Czech Science Foundation (GAČR), project No.\ 22-04381L\@.
A.\,N.\ was partially supported by EPSRC grant EP/X036723/1.

\bibliographystyle{abbrvurl}
\footnotesize{
\bibliography{biblio}}

@article{GLS2019,
    AUTHOR = {Grasegger, Georg and Legersk\'y, Jan and Schicho, Josef},
     TITLE = {Graphs with flexible labelings},
   JOURNAL = {Discrete \& Computational Geometry. An International Journal
              of Mathematics and Computer Science},
    VOLUME = {62},
      YEAR = {2019},
    NUMBER = {2},
     PAGES = {461--480},
      ISSN = {0179-5376,1432-0444},
   MRCLASS = {05C78 (52C25 70B99)},
  MRNUMBER = {3988122},
MRREVIEWER = {Brigitte\ Servatius},
       DOI = {10.1007/s00454-018-0026-9},
       URL = {https://doi.org/10.1007/s00454-018-0026-9},
}

@article{LePf,
    AUTHOR = {Le, Van Bang and Pfender, Florian},
     TITLE = {Extremal graphs having no stable cutsets},
   JOURNAL = {Electronic Journal of Combinatorics},
    VOLUME = {20},
      YEAR = {2013},
    NUMBER = {1},
     PAGES = {\#P35},
      ISSN = {1077-8926},
   MRCLASS = {05C40 (05C35 05C69)},
  MRNUMBER = {3035045},
       DOI = {10.37236/2513},
       URL = {https://doi.org/10.37236/2513},
}

@article{Geiringer,
    author = {H. Pollaczek-Geiringer},
    title = {Über die {G}liederung ebener {F}achwerke},
    journal = {ZAMM. Zeitschrift für Angewandte Mathematik und Mechanik. Journal of Applied Mathematics and Mechanics},
    year = {1927},
    pages  = {58--72},
    volume = {7},
    number = {1},
    doi = {10.1002/zamm.19270070107}
}

@article{Laman,
    AUTHOR = {Laman, G.},
     TITLE = {On graphs and rigidity of plane skeletal structures},
   JOURNAL = {Journal of Engineering Mathematics},
    VOLUME = {4},
      YEAR = {1970},
     PAGES = {331--340},
      ISSN = {0022-0833,1573-2703},
   MRCLASS = {05.40},
  MRNUMBER = {269535},
       DOI = {10.1007/BF01534980},
       URL = {https://doi.org/10.1007/BF01534980},
}

@article{GGLS2021,
    AUTHOR = {Gallet, Matteo and Grasegger, Georg and Legersk\'y, Jan and
              Schicho, Josef},
     TITLE = {On the existence of paradoxical motions of generically rigid
              graphs on the sphere},
   JOURNAL = {SIAM Journal on Discrete Mathematics},
    VOLUME = {35},
      YEAR = {2021},
    NUMBER = {1},
     PAGES = {325--361},
      ISSN = {0895-4801,1095-7146},
   MRCLASS = {05C62 (51F99 53A17 70B99)},
  MRNUMBER = {4228310},
MRREVIEWER = {Alexander\ O.\ Ivanov},
       DOI = {10.1137/19M1289467},
       URL = {https://doi.org/10.1137/19M1289467},
}

@article{Dixon,
	author = {Dixon, A. C.},
	title = {On certain deformable frameworks},
	journal = {The Messenger of Mathematics},
	number = {2},
	pages = {1--21},
	volume = {29},
	year = {1899}
}

@article{Kempe1877,
    AUTHOR = {Kempe, A. B.},
     TITLE = {On {C}onjugate {F}our-piece {L}inkages},
   JOURNAL = {Proceedings of the London Mathematical Society},
    VOLUME = {9},
      YEAR = {1877/78},
     PAGES = {133--147},
      ISSN = {0024-6115},
   MRCLASS = {99-04},
  MRNUMBER = {1575599},
       DOI = {10.1112/plms/s1-9.1.133},
       URL = {https://doi.org/10.1112/plms/s1-9.1.133},
}

@Article{GLSclassification,
    AUTHOR = {Grasegger, Georg and Legersk\'y, Jan and Schicho, Josef},
     TITLE = {On the classification of motions of paradoxically movable
              graphs},
   JOURNAL = {Journal of Computational Geometry},
    VOLUME = {11},
      YEAR = {2020},
    NUMBER = {1},
     PAGES = {548--575},
      ISSN = {1920-180X},
   MRCLASS = {05C62 (52C25 68U05)},
  MRNUMBER = {4194874},
       DOI = {10.20382/jocg.v11i1a22},
       URL = {https://doi.org/10.20382/jocg.v11i1a22},
}

@article{Garamvolgyi2022,
    AUTHOR = {Garamv\"olgyi, D\'aniel},
     TITLE = {Global rigidity of (quasi-)injective frameworks on the line},
   JOURNAL = {Discrete Mathematics},
    VOLUME = {345},
      YEAR = {2022},
    NUMBER = {2},
     PAGES = {Paper No. 112687, 8},
      ISSN = {0012-365X,1872-681X},
   MRCLASS = {05C62 (05C15 68Q25)},
  MRNUMBER = {4334638},
MRREVIEWER = {Brigitte\ Servatius},
       DOI = {10.1016/j.disc.2021.112687},
       URL = {https://doi.org/10.1016/j.disc.2021.112687},
}

@article{ChenYu,
    AUTHOR = {Chen, G. and Yu, X.},
     TITLE = {A note on fragile graphs},
      NOTE = {Combinatorics, graph theory and computing (Louisville, KY,
              1999)},
JOURNAL = {Discrete Mathematics},
    VOLUME = {249},
      YEAR = {2002},
    NUMBER = {1-3},
     PAGES = {41--43},
      ISSN = {0012-365X,1872-681X},
   MRCLASS = {05C69},
  MRNUMBER = {1898256},
       DOI = {10.1016/S0012-365X(01)00226-6},
       URL = {https://doi.org/10.1016/S0012-365X(01)00226-6},
}

@Article{GLSinjective,
   AUTHOR = {Grasegger, Georg and Legersk\'y, Jan and Schicho, Josef},
     TITLE = {Graphs with flexible labelings allowing injective
              realizations},
JOURNAL = {Discrete Mathematics},
    VOLUME = {343},
      YEAR = {2020},
    NUMBER = {6},
     PAGES = {111713, 14},
      ISSN = {0012-365X,1872-681X},
   MRCLASS = {05C62 (05C15)},
  MRNUMBER = {4087212},
MRREVIEWER = {Brigitte\ Servatius},
       DOI = {10.1016/j.disc.2019.111713},
       URL = {https://doi.org/10.1016/j.disc.2019.111713},
}

@article{Maxwell,
    author = {Maxwell, {J. Clerk}},
    title = {On the calculation of the equilibrium and stiffness of frames},
    doi = {10.1080/14786446408643668},
    journal = {The London, Edinburgh, and Dublin Philosophical Magazine and Journal of Science},
    number = {182},
    pages = {294--299},
    volume = {27},
    year = {1864}
}

@article{LeRanderath,
    AUTHOR = {Le, V. B. and Randerath, B.},
     TITLE = {On stable cutsets in line graphs},
   JOURNAL = {Theoretical Computer Science},
    VOLUME = {301},
      YEAR = {2003},
    NUMBER = {1-3},
     PAGES = {463--475},
      ISSN = {0304-3975,1879-2294},
   MRCLASS = {05C40 (68R10)},
  MRNUMBER = {1975241},
       DOI = {10.1016/S0304-3975(03)00048-3},
       URL = {https://doi.org/10.1016/S0304-3975(03)00048-3},
}

@article {chernyshev.etal_2024,
    AUTHOR = {Chernyshev, Vsevolod and Rauch, Johannes and Rautenbach,
              Dieter},
     TITLE = {Forest cuts in sparse graphs},
   JOURNAL = {Discrete Math.},
  FJOURNAL = {Discrete Mathematics},
    VOLUME = {348},
      YEAR = {2025},
    NUMBER = {11},
     PAGES = {Paper No. 114594, 6},
      ISSN = {0012-365X,1872-681X},
   MRCLASS = {05C40 (05C69)},
  MRNUMBER = {4912287},
       DOI = {10.1016/j.disc.2025.114594},
       URL = {https://doi.org/10.1016/j.disc.2025.114594},
}

@article{rauch_2024,
  title = {{R}evisiting {E}xtremal {G}raphs {H}aving {N}o {S}table {C}utsets},
  author = {Rauch, Johannes and Rautenbach, Dieter},
  journal = {Electronic Journal of Combinatorics},
  year={2025},
  volume = {32},
  number = {4},
  pages = {\#P4.25},
  doi = {10.37236/13629}
}

@article{Wsplit,
    AUTHOR = {Whiteley, Walter},
     TITLE = {La division de sommet dans les charpentes isostatiques},
      NOTE = {Dual French-English text},
   JOURNAL = {Structural Topology. Topologie Structurale},
    VOLUME = {16},
      YEAR = {1990},
     PAGES = {23--30},
      ISSN = {0226-9171},
   MRCLASS = {52C25},
  MRNUMBER = {1102001},
MRREVIEWER = {J.\ E.\ Graver},
}

@article{blass.gurevich_1982,
    AUTHOR = {Blass, Andreas and Gurevich, Yuri},
     TITLE = {On the unique satisfiability problem},
   JOURNAL = {Information and Control},
    VOLUME = {55},
      YEAR = {1982},
    NUMBER = {1-3},
     PAGES = {80--88},
      ISSN = {0019-9958},
   MRCLASS = {68Q15 (03B05 03D15 68Q25)},
  MRNUMBER = {727739},
MRREVIEWER = {E.\ Dantzin},
       DOI = {10.1016/S0019-9958(82)90439-9},
       URL = {https://doi.org/10.1016/S0019-9958(82)90439-9},
}

@incollection{jordan_2016,
    AUTHOR = {Jord\'an, Tibor},
     TITLE = {Combinatorial {R}igidity: {G}raphs and {M}atroids in the {T}heory of {R}igid {F}rameworks},
 BOOKTITLE = {Discrete {G}eometric {A}nalysis},
    SERIES = {MSJ Mem.},
    VOLUME = {34},
     PAGES = {33--112},
 PUBLISHER = {Math. Soc. Japan, Tokyo},
      YEAR = {2016},
      ISBN = {978-4-86497-035-8},
   MRCLASS = {52B40 (05B35 05C10 05C62)},
  MRNUMBER = {3525848},
MRREVIEWER = {Brigitte\ Servatius},
}

@article{LeeStreinu,
    AUTHOR = {Lee, Audrey and Streinu, Ileana},
     TITLE = {Pebble game algorithms and sparse graphs},
   JOURNAL = {Discrete Mathematics},
    VOLUME = {308},
      YEAR = {2008},
    NUMBER = {8},
     PAGES = {1425--1437},
      ISSN = {0012-365X,1872-681X},
   MRCLASS = {05C75 (05B35 05C10)},
  MRNUMBER = {2392060},
MRREVIEWER = {Boting\ Yang},
       DOI = {10.1016/j.disc.2007.07.104},
       URL = {https://doi.org/10.1016/j.disc.2007.07.104},
}

@article{AsimowRoth,
    AUTHOR = {Asimow, L. and Roth, B.},
     TITLE = {The rigidity of graphs},
   JOURNAL = {Transactions of the American Mathematical Society},
    VOLUME = {245},
      YEAR = {1978},
     PAGES = {279--289},
      ISSN = {0002-9947,1088-6850},
   MRCLASS = {57M15 (05C10 52A40 53B50 73K05)},
  MRNUMBER = {511410},
MRREVIEWER = {G.\ Laman},
       DOI = {10.2307/1998867},
       URL = {https://doi.org/10.2307/1998867},
}

@article {Coning,
    AUTHOR = {Whiteley, Walter},
     TITLE = {Cones, infinity and 1-story buildings},
   JOURNAL = {Structural Topology},
      YEAR = {1983},
    volume = {8},
     PAGES = {53--70},
    url = {https://hdl.handle.net/2099/1003}
}

@article{EJNSTW,
    AUTHOR = {Eftekhari, Yaser and Jackson, Bill and Nixon, Anthony and
              Schulze, Bernd and Tanigawa, {Shin-ichi} and Whiteley, Walter},
     TITLE = {Point-hyperplane frameworks, slider joints, and rigidity
              preserving transformations},
JOURNAL = {Journal of Combinatorial Theory. Series B},
    VOLUME = {135},
      YEAR = {2019},
     PAGES = {44--74},
      ISSN = {0095-8956,1096-0902},
   MRCLASS = {52C25 (05B35)},
  MRNUMBER = {3926261},
MRREVIEWER = {Brigitte\ Servatius},
       DOI = {10.1016/j.jctb.2018.07.008},
       URL = {https://doi.org/10.1016/j.jctb.2018.07.008},
}

@article {FH96,
    AUTHOR = {Fishburn, Peter C. and Hammer, Peter L.},
     TITLE = {Bipartite dimensions and bipartite degrees of graphs},
JOURNAL = {Discrete Mathematics},
    VOLUME = {160},
      YEAR = {1996},
    NUMBER = {1-3},
     PAGES = {127--148},
      ISSN = {0012-365X},
   MRCLASS = {05C70},
  MRNUMBER = {1417566},
MRREVIEWER = {R. Balakrishnan},
       DOI = {10.1016/0012-365X(95)00154-O},
       URL = {https://doi.org/10.1016/0012-365X(95)00154-O},
}

@misc{flexrilog,
	author = {Grasegger, Georg and Legersk{\'y}, Jan},
	title = {{FlexRiLoG --- SageMath package for Flexible and Rigid Labelings of Graphs}},
	doi = {10.5281/zenodo.3078757},
	howpublished = {Zenodo},
	year = {2019}
}

@misc{minRigidDatabase,
  author       = {Capco, Jose and
                  Gallet, Matteo and
                  Grasegger, Georg and
                  Koutschan, Christoph and
                  Lubbes, Niels and
                  Schicho, Josef},
  title        = {{The number of realizations of all Laman graphs 
                   with at most 12 vertices}},
  year         = 2018,
  howpublished    = {Zenodo},
  doi          = {10.5281/zenodo.1245517},
}

@inproceedings{
	FlexRiLoGPaper,
year = "2020",
	author = "Georg Grasegger and Jan Legerský",
	title = "{FlexRiLoG --- A SageMath Package for Motions of Graphs}",
	booktitle = "{Mathematical Software -- ICMS 2020}",
	editors = "Bigatti, A. and Carette, J. and Davenport, J. and Joswig, M. and de Wolff, T.",
	series = "{Lecture Notes in Computer Science}",
	volume = "12097",
	pages = "442--450",
	publisher = "Springer International Publishing",
	doi = "10.1007/978-3-030-52200-1_44"
}

@misc{Larsson,
	author = { M. Larsson },
	title = {{Nauty Laman plugin}},
	year = {2020},
	publisher = {GitHub},
	howpublished = {\url{https://github.com/martinkjlarsson/nauty-laman-plugin}},
}

@misc{LL2024,
	author = {La\v{s}tovi\v{c}ka, P. and Legerský, J.},
	title = {{Flexible realizations existence: NP-completeness on sparse graphs and algorithms}},
	year = {2024},
	note = {arXiv preprint},
        DOI = {10.48550/arXiv.2412.13721},
}

@misc{pyrigi,
      title = {{PyRigi -- a general-purpose Python package for the rigidity and flexibility of bar-and-joint frameworks}}, 
      author = {Matteo Gallet and Georg Grasegger and Matthias Himmelmann and Jan Legerský},
      year = {2025},
      doi = {10.48550/arXiv.2505.22652}, 
}

@misc{thresholds,
    title={Sharp thresholds for {NAC}-colourings and stable cuts in random graphs},
    author = {Katie Clinch and John Haslegrave and Tony Huynh and Anthony Nixon},
    year={2025},
    NOTE = {arXiv preprint},
    DOI = {10.48550/arXiv.2510.05838},
}
\end{document}